\documentclass[11pt]{article}
\usepackage{latexsym,amsfonts,amssymb,amsmath,amsthm}

\usepackage{color,comment}

\begin{document}

\newtheorem{tm}{Theorem}[section]
\newtheorem{rk}{Remark}[section]
\newtheorem{prop}{Proposition}[section]
\newtheorem{defin}{Definition}[section]
\newtheorem{coro}{Corollary}[section]

\newtheorem{lem}{Lemma}[section]
\newtheorem{assumption}{Assumption}[section]

\newtheorem{nota}[tm]{Notation}
\numberwithin{equation}{section}

\newcommand{\stk}[2]{\stackrel{#1}{#2}}
\newcommand{\dwn}[1]{{\scriptstyle #1}\downarrow}
\newcommand{\upa}[1]{{\scriptstyle #1}\uparrow}
\newcommand{\nea}[1]{{\scriptstyle #1}\nearrow}
\newcommand{\sea}[1]{\searrow {\scriptstyle #1}}
\newcommand{\csti}[3]{(#1+1) (#2)^{1/ (#1+1)} (#1)^{- #1
 / (#1+1)} (#3)^{ #1 / (#1 +1)}}
\newcommand{\RR}[1]{\mathbb{#1}}

\newcommand{ \bl}{\color{blue}}
\newcommand {\rd}{\color{red}}
\newcommand{ \bk}{\color{black}}
\newcommand{ \gr}{\color{OliveGreen}}
\newcommand{ \mg}{\color{RedViolet}}

\newcommand{\ep}{\varepsilon}
\newcommand{\rr}{{\mathbb R}}
\newcommand{\alert}[1]{\fbox{#1}}

\newcommand{\eqd}{\sim}
\newcommand\PP{\ensuremath{\mathbb{P}}}
\def\R{{\mathbb R}}
\def\N{{\mathbb N}}
\def\Q{{\mathbb Q}}
\def\C{{\mathbb C}}
\def\l{{\langle}}
\def\r{\rangle}
\def\t{\tau}
\def\k{\kappa}
\def\a{\alpha}
\def\la{\lambda}
\def\De{\Delta}
\def\de{\delta}
\def\ga{\gamma}
\def\Ga{\Gamma}
\def\ep{\varepsilon}
\def\eps{\varepsilon}
\def\si{\sigma}
\def\Re {{\rm Re}\,}
\def\Im {{\rm Im}\,}
\def\E{{\mathbb E}}
\def\P{{\mathbb P}}
\def\Z{{\mathbb Z}}
\def\D{{\mathbb D}}
\newcommand{\ceil}[1]{\lceil{#1}\rceil}

\title{Long time behavior of random and nonautonomous  Fisher-KPP equations. Part I. Stability of equilibria and spreading speeds}

\author{
Rachidi B. Salako\\
Department of Mathematics\\
The Ohio State University\\
Columbus OH, 43210-1174\\
 and \\
 Wenxian Shen\thanks{Partially supported by the NSF grant DMS--1645673} \\
Department of Mathematics and Statistics\\
Auburn University\\
Auburn University, AL 36849 }

\date{}
\maketitle
\begin{abstract}
\noindent In the current series of two papers, we  study the long time behavior of nonnegative solutions to the following random Fisher-KPP equation,
$$
u_t =u_{xx}+a(\theta_t\omega)u(1-u),\quad x\in\R,
\eqno(1)
$$
where $\omega\in\Omega$, $(\Omega, \mathcal{F},\mathbb{P})$ is a given probability space, $\theta_t$ is  an ergodic metric dynamical system on $\Omega$,  and $a(\omega)>0$ for every $\omega\in\Omega$.
We also study the long time behavior of nonnegative solutions to  the following nonautonomous Fisher-KPP equation,
$$
u_t=u_{xx}+a_0(t)u(1-u),\quad x\in\R,
\eqno(2)
$$
where $a_0(t)$ is a positive locally H\"older continuous function.
In this first part of the series, we investigate the stability of positive equilibria and the spreading speeds.
Under some proper assumption on $a(\omega)$, we show that the constant solution $u=1$ of (1) is asymptotically stable with respect to strictly
positive perturbations and show  that (1) has a deterministic spreading speed interval $[2\sqrt{\underline a}, 2\sqrt{\bar a}]$,
 where $\underline{a}$ and $\bar a$ are the least and the greatest means of $a(\cdot)$, respectively, and hence the spreading speed interval is linearly determinate.  It is shown that
the solution of (1) with a nonnegative initial function which is  bounded away from $0$ for $x\ll -1$  and is  $0$ for  $x\gg 1$ propagates at  the speed $2\sqrt {\hat a}$, where $\hat a$ is the  mean of $a(\cdot)$.
Under some assumption on $a_0(\cdot)$, we also
show that the constant solution $u=1$ of (2) is asymptotically stably and (2) admits a bounded spreading speed interval.
 It is not assumed that $a(\omega)$ and $a_0(t)$ are bounded above and
below by some positive constants. The results obtained in this part are new and extend the existing results in literature on spreading speeds  of Fisher-KPP equations. In the second part of the series, we will study the existence and stability of transition fronts of (1) and (2).
\end{abstract}



\section{Introduction and statements of the main results}

The current series of two papers is concerned with the long time behavior   of  nonnegative solutions to the following random Fisher-KPP equation,
\begin{equation}\label{Main-eq}
u_t =u_{xx}+a(\theta_t\omega)u(1-u),\quad x\in\R,
\end{equation}
where $\omega\in\Omega$,  $(\Omega, \mathcal{F},\PP,\{\theta_t\}_{t\in \R})$ is an ergodic  metric dynamical system  on $\Omega$,
 $a:\Omega\to (0,\infty)$ is measurable, and $a^\omega(t):=a(\theta_t\omega)$ is locally H\"older continuous for every $\omega\in\Omega$.
It also considers the long time behavior   of nonnegative solutions to  the following nonautonomous  Fisher-KPP equation,
\begin{equation}\label{nonautonomous-eq}
u_t=u_{xx}+a_0(t)u(1-u), x\in\R,
\end{equation}
where $a_0:\R\to (0,\infty)$ is locally H\"older continuous.
 Among others,  \eqref{Main-eq} and \eqref{nonautonomous-eq} are used to model  the  population growth of a species in biology. In such case,  $u(t,x)$ denotes the population density  of the species. Thanks to the biological reason, we are
only interested in nonnegative solutions of \eqref{Main-eq} and \eqref{nonautonomous-eq}.

Observe that \eqref{Main-eq} (resp. \eqref{nonautonomous-eq}) with $a(\omega)\equiv 1$ (resp. with $a_0(t)\equiv 1$) becomes
\begin{equation}
\label{fisher-kpp}
u_t=u_{xx}+u(1-u),\quad x\in\R.
\end{equation}
Equation \eqref{fisher-kpp}  is called in literature  Fisher-KPP   equation
 due to the pioneering works of Fisher
\cite{Fisher} and Kolmogorov, Petrowsky, Piskunov \cite{KPP} on traveling wave solutions and take-over properties of \eqref{fisher-kpp}.
It is clear that the constant solution $u=1$ of \eqref{fisher-kpp} is asymptotically stable with respect to strictly positive perturbations.
Fisher in
\cite{Fisher} found traveling wave solutions $u(t,x)=\phi(x-ct)$ of \eqref{fisher-kpp}
$(\phi(-\infty)=1,\phi(\infty)=0,  \phi(s)>0)$ of all speeds $c\geq 2$ and
showed that there are no such traveling wave solutions of slower
speed. He conjectured that the take-over occurs at the asymptotic
speed $2$. This conjecture was proved in \cite{KPP}  {for some special initial distribution and was proved in \cite{ArWe2} for general initial distributions.
 More precisely, it is proved
in \cite{KPP} that for the  nonnegative solution $u(t,x)$ of \eqref{fisher-kpp} with
$u(0,x)=1$ for $x<0$ and $u(0,x)=0$ for $x>0$, $\lim_{t\to \infty}u(t,ct)$ is $0$ if $c>2$ and $1$ if $c<2$. It is proved
in \cite{ArWe2} that for any
nonnegative solution $u(t,x)$ of (\ref{fisher-kpp}), if at
time $t=0$, $u$ is $1$ near $-\infty$ and $0$ near $\infty$, then
$\lim_{t\to \infty}u(t,ct)$ is $0$ if $c>2$ and $1$ if $c<2$.}
In
literature, $c^*=2$ is   called the {\it
spreading speed} for \eqref{fisher-kpp}.

A huge amount of research has been carried out toward various extensions of
 traveling wave solutions and take-over properties  of \eqref{fisher-kpp} to general time and space independent
as well as time and/or space dependent Fisher-KPP type  equations.
See, for example,  \cite{ArWe1, ArWe2, Bra, Ham, Kam, Sat, Uch}, etc., for
the extension to general time and space independent Fisher-KPP type equations;  see
 \cite{BeHaRoq, BeHaNa1,  BeHaRo,  FrGa,   HuZi1, KoSh, LiYiZh, LiZh, LiZh1,  Nad,  NoRuXi, NoXi,  Wei1, Wei2},  and references therein for
the extension to time and/or space periodic Fisher-KPP
type equations; and see
\cite{BeHaNa1, BeHa07, BeHa12, BeNa,  HePaSt, HuSh,  Mat, Nad1,  NaRo1, NaRo2, NaRo3, NoRoRyZl,  She6, She7, She8,  She9, TaZhZl, Xin1, Zla}, and references therein for
the extension to quite general time and/or space dependent Fisher-KPP
type equations.  The reader is referred to \cite{FaHuWu}, \cite{HeWu}, \cite{ZoWu}, etc.
   for the study of Fisher-KPP reaction diffusion equations with time delay.

 All the existing works on \eqref{Main-eq} (resp. \eqref{nonautonomous-eq})  assumed $\inf_{t\in\R} a^\omega(t)>0$ and $a^\omega(\cdot)\in L^\infty(\R)$ (resp. $\inf_{t\in\R} a_0(t)>0$ and $\sup_{t\in\R} a_0(t)<\infty$).
The objective of the current series of two papers  is to study the long time behavior, in particular, the stability of positive constant solutions, the spreading speeds, and the transition fronts  of \eqref{Main-eq} (resp. \eqref{nonautonomous-eq}) without the assumption $\inf_{t\in\R} a^\omega(t)>0$ and $a^\omega(\cdot)\in L^\infty(\R)$ (resp. without the assumption $\inf_{t\in\R} a_0(t)>0$ and $\sup_{t\in\R} a_0(t)<\infty$). It will also discuss the applications of the results established for \eqref{Main-eq} to Fisher-KPP equations whose growth rate and/or carrying
 capacity are perturbed by real noises.

In this first part of the series, we investigate the stability of positive constant solutions and the spreading speeds of \eqref{Main-eq} and
\eqref{nonautonomous-eq}.
We first consider the stability of positive constant solutions and spreading speeds  of \eqref{Main-eq}
and then consider the stability of positive constant solutions and spreading speeds of \eqref{nonautonomous-eq}.
In the second part of the series, we will study the existence and stability of transition fronts of \eqref{Main-eq} and
\eqref{nonautonomous-eq}.

In the following, we  state the main results of the current  paper.   Let
$$
C_{\rm unif}^b(\R)=\{u\in C(\R)\,|\, u\,\, \text{is bounded and uniformly continuous}\}
$$
with norm $\|u\|_\infty=\sup_{x\in\R}|u(x)|$ for $u\in C_{\rm unif}^b(\R)$.
For given $u_0\in X:=C_{\rm unif}^b(\R)$ and $\omega\in\Omega$, let $u(t,x;u_0,\omega)$ be the solution of \eqref{Main-eq} with $u(0,x;u_0,\omega)=u_0(x)$.
Note that, for $u_0\in X$ with $u_0\ge 0$, $u(t,x;u_0,\omega)$  exists  for $t\in [0,\infty)$ and $u(t,x;u_0,\omega)\ge 0$ for all $t\ge 0$.
Note also that $u\equiv 0$ and $u\equiv 1$ are two constant solutions of \eqref{Main-eq}.
 Let
\begin{equation}\label{a-least-mean}
{ \hat{a}_{\inf}(\omega)}=\liminf_{t-s\to\infty}
\frac{1}{t-s}\int_s^ta(\theta_{\tau}\omega)d\tau:=\lim_{r\to\infty}\inf_{t-s\ge r}\frac{1}{t-s}\int_s^ta(\theta_{\tau}\omega)d\tau
\end{equation}
and
\begin{equation}
\label{a-largest-mean}
{ \hat{a}_{\sup}(\omega)}=\limsup_{t-s\to\infty}\frac{1}{t-s}\int_s^ta(\theta_\tau\omega)d\tau:=\lim_{r\to\infty}\sup_{t-s\ge r}\frac{1}{t-s}\int_s^ta(\theta_{\tau}\omega)d\tau.
\end{equation}

 Observe that
\begin{equation}
\label{a-a-eq1}
\hat{a}_{\inf}(\theta_t\omega)=\hat{a}_{\inf}(\omega)\quad {\rm and}\quad \hat{a}_{\sup}(\theta_t\omega)=\hat{a}_{\sup}(\omega),\,\, \forall\,\, t\in\R,
\end{equation}
and that
$$
\hat{a}_{\inf}(\omega)=\liminf_{t,s\in\Q,t-s\to\infty}\frac{1}{t-s}\int_s^t a(\theta_\tau)d\tau\,\, \,\, {\rm and}\,\, \,\, \hat{a}_{\sup}(\omega)=\liminf_{t,s\in\Q,t-s\to\infty}\frac{1}{t-s}\int_s^t a(\theta_\tau)d\tau.
$$
Then by the countability of the set $\Q$ of rational numbers, both $\hat{a}_{\inf}(\omega)$ and $\hat{a}_{\sup}(\omega)$ are measurable in $\omega$.

Throughout this paper, we assume that  the following standing assumption holds.
\medskip

\noindent {\bf (H1)}
{\it  $0< \hat{a}_{\inf}(\omega)\le \hat{a}_{\sup}(\omega)<\infty$ for a.e. $\omega\in\Omega$.}

\medskip
Note that
  {\bf (H1)} implies that $\hat{a}_{\inf}(\cdot),a(\cdot),\hat{a}_{\sup}(\cdot)\in  L^1 (\Omega, \mathcal{F},\PP)$, {  and that }
 there are
 $\hat a, \underline{a}, \bar a\in\R^+$ and a measurable subset $\Omega_0\subset\Omega$ with $\P(\Omega_0)=1$ such that
 \begin{equation}
 \label{omega-0}
 \begin{cases}
 \theta_t\Omega_0=\Omega_0\quad \forall\,\, t\in\R,\cr
 \lim_{t\to \pm \infty}\frac{1}{t}\int_0^t a(\theta_\tau\omega)d\tau=\hat a\quad \forall\,\, \omega\in\Omega_0,\cr
 { \hat a_{\inf}(\omega)} =\underline{a}\quad \forall\,\,\omega\in\Omega_0,\cr
 { \hat a_{\sup}(\omega)}=\bar a\quad\forall\,\, \omega\in\Omega_0
 \end{cases}
 \end{equation}
 {(see Lemma \ref{time-avg-vs-space-avg-lemma}).}  Throughout this paper,
 $\hat a$ is referred to as the {\it  mean} or {\it average} of $a(\cdot)$, and
$\underline{a}$ and $\overline {a}$ are referred to as the {\it least mean} and the {\it greatest mean} of $a(\cdot)$, respectively.

 Our main result on the stability of the constant solution $u\equiv 1$ of \eqref{Main-eq} reads as follows.

\begin{tm}\label{stability of const equi solu thm}
For every $u_0\in C^{b}_{\rm uinf}(\R)$ with {$\inf_{x\in\R}u_0(x)>0$} and for every $\omega\in\Omega$, we  have that
\begin{equation}\label{stability of const equi solu eq1}
\|u(t,\cdot;u_0,\omega)-1\|_{\infty}\leq M(u_0) e^{-\int_0^ta(\theta_s\omega)ds},
\end{equation}
where $M(u_0):=\max\{1,\|u_0\|_{\infty}\}\cdot\max\Big\{\Big|1-\frac{1}{\min\{1,\inf_{x\in\R} u_0(x)\}} \Big|, \Big|1-\frac{1}{\max\{1,\sup_{x\in\R} u_0(x)\}} \Big|\Big\}$. Hence if $\int_0^{\infty}a(\theta_s\omega)ds=\infty$, then
$$
\lim_{t\to\infty} \|u(t,\cdot;u_0,\omega)-1\|_{\infty}=0.
$$
In particular, if {\bf (H1)} holds, then for every $0<\tilde{a}<\underline{a}$,  every $u_0\in C^{b}_{\rm uinf}(\R)$ with $\inf_{x}u_0(x)>0$, and almost all $\omega\in\Omega$, there is positive constant $M>0$  such that
$$
\|u(t,\cdot;u_0,\theta_{t_0}\omega)-1\|_{\infty}\leq Me^{-\tilde{a}t},\quad \forall \ t\geq 0, \ t_0\in\R.
$$
If $a(\theta_{\cdot}\omega)\in L^{1}(0,\infty)$, then the constant equilibrium solution, $u\equiv 1$, of \eqref{Main-eq} is not asymptotically stable.
 \end{tm}

 To state our main results on the spreading speeds of \eqref{Main-eq},
 let
\begin{equation}
\label{minimal-speeds-eq}
\underline{c}^*=2\sqrt {\underline{a}},\quad  \hat c^*=2\sqrt{\hat{a}},\quad   {\rm and}\quad \overline{c}^*=2\sqrt {\overline{a}}.
\end{equation}
  Let
 $$
 X_c^+=\{u\in C_{\rm unif}^b(\R)\,|\, u\ge 0,\,\,  {\rm supp}(u)\,\,\, \text{is bounded and not empty}\}.
 $$

 \begin{defin}
 \label{spreading-speed-def}
 For given $\omega\in\Omega$, let
 $$
 C_{\sup}(\omega)=\{c\in\R^+\,|\, \limsup_{t\to\infty}\sup_{s\in\R,|x|\ge ct}u(t,x;u_0,\theta_s\omega)=0\quad \forall\,\, u_0\in X_c^+\}
 $$
 and
 $$
 C_{\inf}(\omega)=\{c\in\R^+\,|\, \limsup_{t\to\infty}\sup_{s\in\R,|x|\le ct} |u(t,x;u_0,\theta_s\omega)-1|=0\quad \forall\,\, u_0\in X_c^+\}.
 $$
 Let
 $$
 c_{\sup}^*(\omega)=\inf\{c\,|\, c\in C_{\sup}(\omega)\},\quad c_{\inf}^*(\omega)=\sup\{c\,|\, c\in C_{\inf}(\omega)\}.
 $$
 $[c_{\inf}^*(\omega),c_{\sup}^*(\omega)]$ is called the {\rm spreading speed interval} of \eqref{Main-eq}  with respect to compactly supported initial functions.
 \end{defin}

{
The following theorem   shows that the spreading speed interval of \eqref{Main-eq} with respect to compactly supported initial functions is deterministic and is linearly {determinate}, that is,  $ [c_{\inf}^*(\omega),c_{\sup}^*(\omega)]=[\underline{c}^*,\bar{c}^*]$ for all $\omega\in\Omega_0$.

\begin{tm}\label{spreading-speeds-tm}
Assume that {\bf (H1)} holds. Then the following hold.

\begin{itemize}
\item[(i)] For any $\omega\in\Omega_0$, $c_{\sup}^*(\omega)=\bar{c}^*$.

\item[(ii)]  For any $\omega\in\Omega_0$,  $c_{\inf}^{*}(\omega)=\underline{c}^*$.
\end{itemize}
\end{tm}

The above theorem concerns the spreading speeds of solutions of \eqref{Main-eq} with  compactly supported nonnegative initial functions.
To consider  the spreading speeds of solutions of \eqref{Main-eq} with front-like initial functions, let
 $$
\tilde  X_c^+=\{u\in C_{\rm unif}^b(\R)\,|\, u\ge 0,\,\,  \liminf_{x\to -\infty}u_0(x)>0,\,\, u_0(x)=0\,\, {\rm for}\,\, x\gg 1\}.
 $$

 \begin{defin}
 \label{spreading-speed-def1}
 For given $\omega\in\Omega$, let
 $$
 \tilde C_{\sup}(\omega)=\{c\in\R^+\,|\, \limsup_{t\to\infty}\sup_{s\in\R,x\ge  ct}u(t,x;u_0,\theta_s\omega)=0\quad \forall\,\, u_0\in \tilde X_c^+\}
 $$
 and
 $$
 \tilde C_{\inf}(\omega)=\{c\in\R^+\,|\, \limsup_{t\to\infty}\sup_{s\in\R,x\le ct} |u(t,x;u_0,\theta_s\omega)-1|=0\quad \forall\,\, u_0\in \tilde X_c^+\}.
 $$
 Let
 $$
 \tilde c_{\sup}^*(\omega)=\inf\{c\,|\, c\in \tilde C_{\sup}(\omega)\},\quad \tilde c_{\inf}^*(\omega)=\sup\{c\,|\, c\in \tilde C_{\inf}(\omega)\}.
 $$
 $[\tilde c_{\inf}^*(\omega),\tilde c_{\sup}^*(\omega)]$ is called the {\rm spreading speed interval} of \eqref{Main-eq}  with respect to front-like initial functions.
 \end{defin}

 We have the following theorem on the spreading speeds  of the solutions with front-like initial functions.

\begin{tm}\label{spreading-speeds-tm-0}
Assume that {\bf (H1)} holds. Then the following hold.

\begin{itemize}
\item[(i)] For any $\omega\in\Omega_0$, $\tilde c_{\sup}^*(\omega)=\bar{c}^*$.

\item[(ii)]  For any $\omega\in\Omega_0$,  $\tilde c_{\inf}^{*}(\omega)=\underline{c}^*$.
\end{itemize}
\end{tm}

 We also have the following  theorem on the take-over property of the solutions of \eqref{Main-eq} with front-like initial functions  and   with the initial function
$u_0^*(x)=1$ for $x< 0$ and $u_0^*(x)=0$ for $x>0$.  Note that $u(t,x;u_0^*,\omega)$ exists for all $t>0$ (see \cite[Theorem 1]{KPP}).

\begin{tm}\label{spreading-speeds-tm-1}
\begin{description}
\item[(i)]
 For a.e. $\omega\in\Omega$,
\begin{equation}
\label{average-speed-eq}
\lim_{t\to\infty}\frac{x(t,\omega)}{t}= \hat c^*,
\end{equation}
where $x(t,\omega)$ is such that $u(t,x(t,\omega);u_0^*,\omega)=\frac{1}{2}$.  Moreover,
\begin{equation}\label{asymptoc-tail-eq1}
\lim_{t\to\infty}\sup_{x\geq(\hat c^*+h)t}u(t,x;u_0^*,\omega)=0, \forall\ h>0, \ \text{a.e }\ \omega
\end{equation}
and
\begin{equation}\label{asymptoc-tail-eq2}
\lim_{t\to\infty}\inf_{x\leq (\hat c^*-h)t}u(t,x;u_0^*,\omega)=1, \forall\ h>0, \ \text{a.e\  }\ \omega.
\end{equation}

\item[(ii)]   For any $u_0\in \tilde X_c^+$, it holds that
\begin{equation}\label{asymptoc-tail-eq3}
\lim_{t\to\infty}\sup_{x\geq(\hat c^*+h)t}u(t,x;u_0,\omega)=0, \forall\ h>0, \ \text{a.e }\ \omega
\end{equation}
and
\begin{equation}\label{asymptoc-tail-eq4}
\lim_{t\to\infty}\inf_{x\leq (\hat c^*-h)t}u(t,x;u_0,\omega)=1, \forall\ h>0, \ \text{a.e\  }\ \omega.
\end{equation}
\end{description}

\end{tm}

\medskip

Consider now \eqref{nonautonomous-eq}. { Define $\underline{a}_0$ and $\overline{a}_0$ by
\begin{equation}
\label{b-b-eq}
\underline{a}_0= \liminf_{t-s\to\infty}\frac{1}{t-s}\int_s^t a_0(\tau)d\tau,\quad \overline{a}_0=\limsup_{t-s\to\infty}\frac{1}{t-s}\int_s^t a_0(\tau)d\tau.
\end{equation}
}

 \noindent Let { (H2)} be the following standing assumption.

\medskip

\noindent {\bf (H2)}  {\it  $0< \underline{a}_0\le \overline{a}_0<\infty$.}

\medskip

The assumption {\bf (H2)} is the analogue of {\bf (H1)}. We will give some example for $a_0(\cdot)$ {satisfying {\bf (H2)}} in section 5.
Assume  {\bf (H2)}.  Let
\begin{equation}
\bar c_0^*=2\sqrt{\bar{a}_0}\quad {\rm and}\quad \underline{c}_0^*=2\sqrt {\underline{a}_0}.
\end{equation}
For given $u_0\in C_{\rm unif}^b(\R)$ with $u_0\ge 0$ and $s\in\R$, let $u(t,x;u_0,\sigma_s a_0)$ be the solution of
$$
u_t=u_{xx}+\sigma_s a_0(t) u(1-u),\quad x\in\R,\, t>0,
$$
with $u(0,x;u_0,\sigma_s a_0)=u_0(x)$, where $\sigma_s a_0(t)=a_0(s+t)$.

We have the following theorem on the spreading speeds of \eqref{nonautonomous-eq}.

\begin{tm}\label{nonautonomous-thm3}
Assume  {\bf (H2)}.
Then for every $u_0\in X_c^+$,
\begin{equation}\label{nonauton-spreading-speed-eq1}
\liminf_{t\to\infty}\sup_{s\in\R,|x|\le ct}|u(t,x;u_0,\sigma_s a_0)-1|=0, \quad \forall\ 0<c<\underline{c}^*_0:=2\sqrt{\underline{a}_0}
\end{equation}
and
\begin{equation}\label{nonauton-spreading-speed-eq2}
\limsup_{t\to\infty}\sup_{s\in\R,|x|\ge ct} u(t,x;u_0,\sigma_s a_0)=0, \quad \forall\ c>\bar c^*_0:=2\sqrt{\bar a_0}.
\end{equation}
\end{tm}

\medskip

We conclude the introduction with the following  four remarks.

 First, the results in Theorems \ref{spreading-speeds-tm}-\ref{nonautonomous-thm3} are new. If $a_0(t)$ is periodic with period $T$, then $\underline{a}_0=\bar a_0=\hat a_0:=\frac{1}{T}\int_0^T a_0(\tau)d\tau$ and hence
 $\underline{c}_0^*=\bar c_0^*=2\sqrt{\hat a_0}$. More generally, {if $a_0(t)$ in globally H\"older continuous and
  is uniquely ergodic in the sense that
 the space $H(a_0)$ is compact and the flow $(H(a_0),\sigma_t)$ is uniquely ergodic,} where $H(a_0)={\rm cl}\{\sigma_s a_0\,|\, s\in\R\}$ with open compact topology and
 $\sigma_s a_0(\cdot)=a_0(\cdot+s)$,
 then $\underline{a}_0=\bar a_0=\hat a_0:=\lim_{T\to\infty}\frac{1}{T}\int_0^T a_0(\tau)d\tau$ and hence   $\underline{c}_0^*=\bar c_0^*=2\sqrt{\hat a_0}$. Therefore the existing results on spreading speeds of \eqref{nonautonomous-eq} in the time periodic and time almost
 periodic cases are recovered. The current paper provides a new and simpler proof in these special cases.

\smallskip

 Second, by Theorem \ref{spreading-speeds-tm} and \ref{spreading-speeds-tm-0},
 $$
 [c_{\rm inf}^*(\omega),c_{\rm sup}^*(\omega)]=[\tilde c_{\rm inf}^*(\omega),\tilde c_{\rm sup}^*(\omega)]=[\underline{c}^*,\bar c^*]
 $$
 for any $\omega\in\Omega_0$.  Hence $[\underline{c}^*,\bar c^*]$ is called the {\it spreading speed interval} of \eqref{Main-eq},
 which is deterministic and is determined by the linearized equation of \eqref{Main-eq} at $u\equiv 0$. Theorem \ref{spreading-speeds-tm-1} is an extension of the take-over property  proved in  \cite{ArWe2}  and  \cite{KPP} for \eqref{fisher-kpp}.   In order to prove Theorem \ref{spreading-speeds-tm-1} we are first led to prove that $x(t,\omega) $ is a subadditive process (see Lemma \ref{lem2} for more detail). The fact that $x(t,\omega)$ is a subadditive process is    interesting. Its proof  relies on comparison between various translation of the solution and on a zero-number argument enabling to bound the width of the interface. {It is our belief that} this result will open the way to other applications in the future.

\smallskip

 Third,  the results established for \eqref{Main-eq} and \eqref{nonautonomous-eq} can be applied to the following general random Fisher-KPP equation,
\begin{equation}
\label{general-random-eq}
u_t=u_{xx}+u(r(\theta_t\omega)-\beta(\theta_t\omega) u),
\end{equation}
where $r:\omega\to (-\infty,\infty)$ and $\beta:\Omega\to (0,\infty)$ are measurable with locally H\"older continuous sample paths $r^\omega(t):=r(\theta_t\omega)$ and
$\beta^\omega(t):=\beta(\theta_t\omega)$,
and to the following nonautonomous Fisher-KPP equation,
\begin{equation}
\label{general-nonautonomous-eq}
u_t=u_{xx}+u(r_0(t)-\beta_0(t) u),
\end{equation}
where $r_0:\R\to \R$ and $\beta_0:\R\to (0,\infty)$ are locally H\"older continuous. Note that \eqref{general-random-eq} models the population growth of a species with random
perturbations on its growth rate and carrying capacity, and \eqref{general-nonautonomous-eq} models the population growth of a species
with deterministic time dependent perturbations on its growth rate and carrying capacity.

In fact, under some assumptions on $r(\omega)$ and $\beta(\omega)$, it can be proved that
$$
u(t;\omega):=Y(\theta_t\omega)=\frac{1}{\int_{-\infty}^0 e^{-\int_s ^0 r(\theta_{\tau+t}\omega)d\tau}\beta(\theta_{s+t}\omega)ds}
$$
is an random equilibrium of \eqref{general-random-eq}. Let $\tilde u=\frac{u}{Y(\theta_t\omega)}$ and drop the tilde, \eqref{general-random-eq} becomes
\eqref{Main-eq} with $a(\theta_t\omega)=\beta(\theta_t\omega)\cdot Y(\theta_t\omega)$, and then the results established for \eqref{Main-eq}
can be applied to \eqref{general-random-eq}.
For example, consider the following random Fisher-KPP equation,
\begin{equation}
\label{real-noise-eq}
u_t=u_{xx}+u(1+ \xi(\theta_t\omega) -u),\quad x\in\R,
\end{equation}
 where $\omega\in\Omega$,  $(\Omega, \mathcal{F},\PP,\{\theta_t\}_{t\in \R})$ is an ergodic  metric dynamical system,  $\xi:\Omega\to \R$ is measurable,  and $\xi_t(\omega):=\xi(\theta_t\omega)$ is locally H\"older continuous
($\xi_t$ denotes a real noise or a colored noise). { Let $\hat\xi_{\inf}(\omega)$ and $\hat\xi_{\sup}(\omega)$ be defined as in
\eqref{a-least-mean} and \eqref{a-largest-mean} with $a(\cdot)$ being replaced by $\xi(\cdot)$, respectively}.
Assume that $\xi_t(\cdot)$ satisfies the following {\bf (H3)}.

\medskip

\noindent  {\bf (H3)}  {\it $\xi:\Omega\to\R$ is measurable;  $\int_\Omega |\xi(\omega)|d\PP(\omega)<\infty$ and
$\int_\Omega \xi(\omega)d\PP(\omega)=0$;   $-1<{ \hat {\xi}_{\inf}(\omega)}\le {\hat {\xi}_{\sup}(\omega)}<\infty$ and { ${ \inf_{t\in\R}\xi(\theta_{t}\omega)}>-\infty$} for a.e. $\omega\in\Omega$;   and
 $\xi^\omega(t):=\xi(\theta_t\omega)$ is locally H\"older continuous.  }

 \medskip

Assume {\bf (H3)}. By the arguments of Lemma \ref{time-avg-vs-space-avg-lemma}, there are $\underline{\xi},\overline{\xi}\in\R$ such that
$\hat\xi_{\inf }(\omega)=\underline\xi$ and $\hat\xi_{\sup}(\omega)=\overline\xi$ for a.e. $\omega\in\Omega$.
 It  can be proved that
 \begin{equation}
 \label{random-equilibrium-1}
 Y(\omega)=\frac{1}{\int_{-\infty}^0 e^{ s+\int_0^s \xi(\theta_\tau\omega)d\tau}ds}
 \end{equation}
  is a spatially homogeneous asymptotically stable
 random equilibrium of \eqref{real-noise-eq} (see Theorem \ref{real-noise-tm1} and Corollary \ref{stability-random-equilibrium-cor}). It can also be proved that for any $u_0\in X_c^+$,
 \begin{equation*}
\limsup_{t\to\infty}\sup_{s\in\R,|x|\le ct}|\frac{u(t,x;u_0,\theta_s\omega)}{Y(\theta_{t+s}\omega)}-1|=0, \quad \forall\ 0<c<2\sqrt{1+\underline{\xi}}
\end{equation*}
and
\begin{equation*}
\limsup_{t\to\infty}\sup_{s\in\R,|x|\ge ct} \frac{u(t,x;u_0,\theta_s\omega)}{Y(\theta_{t+s}\omega)}=0, \quad \forall\ c>2\sqrt {1+\bar \xi},
\end{equation*}
for a.e. $\omega\in\Omega$. where $u(t,x;u_0,\theta_s\omega)$ is the solution of \eqref{real-noise-eq}
with $\omega$ being replaced by $\theta_s\omega$ and $u(0,x;u_0,\theta_s\omega)=u_0(x)$ (see Corollary \ref{spreading-cor}).

\medskip

 Fourth, it is interesting to study the spreading properties of \eqref{Main-eq} with {\bf (H1)} being replaced by the following weaker assumption,
\medskip

\noindent {\bf (H1$)'$} $0<\hat a:=\int_\Omega a(\omega)d\PP(\omega)<\infty$.

\medskip

\noindent We plan to study this general case somewhere else, which would have applications to the study of the spreading properties of
the following stochastic Fisher-KPP equation,
\begin{equation}
\label{white-noise-eq}
d u=(u_{xx}+u(1-u))dt+\sigma u dW_t,\quad x\in\R,
\end{equation}
where $W_t$ denotes the standard two-sided Brownian motion ($dW_t$ is then the white noise).
In fact, let $ \Omega:=\{\omega\in C(\R,\R)\ |\  \omega(0)=0\ \}$ equipped with the open compact topology, $\mathcal{F}$ be the Borel $\sigma-$field and $\mathbb{P}$ be the Wiener measure on $(\Omega, \mathcal{F})$. Let $W_t$ be the one dimensional Brownian motion on the Wiener space $(\Omega,\mathcal{F},\mathbb{P})$ defined by $W_t(\omega)=\omega(t)$. Let $\theta_t\omega$ be the canonical Brownian shift: $(\theta_t\omega)(\cdot)=\omega(t+\cdot)-\omega(t)$ on $\Omega$. It is easy to see that $W_t(\theta_s\omega)=W_{t+s}(\omega)-W_s(\omega)$.
  If $\frac{\sigma^2}{2}<1$,
then it can be proved that
\begin{equation}
\label{random-equilibrium-eq2}
Y(\omega)=\frac{1}{\int_{-\infty}^0 e^{(1-\frac{\sigma^2}{2})s+\sigma W_s (\omega)ds}}
\end{equation}
 is a spatially homogeneous stationary solution process of \eqref{white-noise-eq}. Let $\tilde u=\frac{u}{Y(\theta_t\omega)}$ and drop the tilde, \eqref{white-noise-eq} becomes
\eqref{Main-eq} with $a(\theta_t\omega)= Y(\theta_t\omega)$. The reader is referred to \cite{HuLi1, HuLi2, HuLiWa, JiJiSh, OkVaZh1, OkVaZh2}
for some study on the front propagation dynamics of \eqref{random-equilibrium-eq2}.  Note that Theorem \ref{spreading-speeds-tm-1} (i)
is an analogue of \cite[Theorem 1]{HuLiWa}.

 It is important to note that the authors of {the work \cite{BeNa} studied } the asymptotic spreading speeds for space-time heterogeneous equations of the form
\begin{equation}\label{BeNa}
u_t=\sum_{i,j=1}^Na_{i,j}(t,x)u_{x_ix_j} +\sum_{i=1}^Nq_i(t,x)u_{x_i} +f(t,x,u),\quad x\in\R^N,
\end{equation}
where $f(t,x,0)=f(t,x,1)=0$. We note that Theorem \ref{nonautonomous-thm3} improves \cite[Proposition 3.9]{BeNa}, since  $\inf_{t\in\R}a_0(t)=0$ and $\sup_{t\in\R}a_0(t)$ are allowed here. Moreover, the techniques developed in the current work are different from the ones in \cite{BeNa}. Certainly, it should be mentioned that \eqref{BeNa} is more general than \eqref{nonautonomous-eq}.

The rest of the paper is organized as follows. In section 2, we  present some preliminary lemmas, which will be used in the proofs of main
results of the current paper  in later sections.
   In section 3, we establish some results about the stability of the positive constant equilibrium solution $u\equiv 1$ of \eqref{Main-eq} (resp. \eqref{nonautonomous-eq}) {and prove Theorem \ref{stability of const equi solu thm}}.
In section 4, we study the spreading properties of solutions of \eqref{Main-eq} with nonnegative and compactly supported initial functions
or front like initial functions and prove Theorems \ref{spreading-speeds-tm} and \ref{spreading-speeds-tm-0}.
We investigate in section 4 the take-over property of \eqref{Main-eq} and prove Theorem \ref{spreading-speeds-tm-1}.
We consider spreading properties of \eqref{nonautonomous-eq} in section 5.

\section{Preliminary lemmas}

In this section, we present some preliminary lemmas to be used in later sections of this paper as well as in the second part of the series.

\begin{lem}
\label{time-avg-vs-space-avg-lemma}
  {\bf (H1)} implies that $\hat{a}_{\inf}(\cdot),a(\cdot), { \hat a_{\sup}(\cdot)}\in  L^1 (\Omega, \mathcal{F},\PP)$ and that
 {there are
 $\underline{a}, \bar a, \hat a\in\R^+$ and a measurable subset $\Omega_0\subset\Omega$ with $\P(\Omega_0)=1$ such that
 $\theta_t\Omega_0=\Omega_0$ for all $t\in\R$,  $\hat a_{\inf}(\omega) =\underline{a}$ and
 $\hat a_{\sup}(\omega)=\bar a$ for all $\omega\in\Omega_0$, and
 $\lim_{t\to \pm \infty}\frac{1}{t}\int_0^t a(\theta_\tau\omega)d\tau=\hat a$ for all $\omega\in\Omega_0$.}
\end{lem}

\begin{proof} First,  let
 $$
 \Omega_n=\{\omega\in\Omega\, |\, \hat a_{\sup}(\omega)\le n\}\quad \forall\, n\in\N,
 $$
  and
 $$
 \Omega_\infty=\{\omega\in\Omega\,|\, \hat a_{\sup}(\omega)=\infty\}.
  $$
Then ${\Omega_\infty\cup} \cup_{n=1}^\infty \Omega_n=\Omega$. {By {\bf (H1)}},  there is $\bar n\in\N$ such that
 $\P(\Omega_{\bar n})>0$. By \eqref{a-a-eq1},  $\theta_t\Omega_n=\Omega_n$ for all $t\in\R$ and $n\in\N$.  Then by the ergodicity of the metric dynamical system $(\Omega, \mathcal{F},\PP,\{\theta_t\}_{t\in \R})$,
 we have $\PP(\Omega_{\bar n})=1$.  This implies that $\hat{a}_{\sup}(\cdot)\in  L^1 (\Omega, \mathcal{F},\PP)$, and then $\hat{a}_{\inf}(\cdot)\in  L^1 (\Omega, \mathcal{F},\PP)$. Moreover,  by \eqref{a-a-eq1},
 \begin{equation*}
\hat{a}_{\inf}(\omega)=\lim_{t\to\infty}\frac{1}{t}\int_0^t\hat{a}_{\inf}(\theta_\tau\omega)d\tau=\int_\Omega\hat{a}_{\inf}(\omega) d\PP(\omega)\quad {\rm for}\quad a.e. \,\, \omega\in\Omega,
\end{equation*}
and
\begin{equation*}
\hat{a}_{\sup}(\omega)=\lim_{t\to\infty}\frac{1}{t}\int_0^t\hat{a}_{\sup}(\theta_\tau\omega)d\tau=\int_\Omega\hat{a}_{\sup}(\omega) d\PP(\omega)\quad {\rm for}\quad a.e. \,\, \omega\in\Omega.
\end{equation*}
It then follows that there are $\underline a,\bar a\in\R$ and  a measurable subset $\Omega_1\subset\Omega$ with $\P(\Omega_1)=1$ such that
$\theta_t\Omega_1=\Omega_1$ for all $t\in\R$, and $\hat a_{\inf}(\omega)=\underline a$ and
$\hat a_{\sup}(\omega)=\bar a$ for all $\omega\in\Omega_1$.

Next, for given $n\in\N$, let
$$
a_n(\omega)=\min\{a(\omega), n\}.
$$
Then $a_n(\cdot)\in  L^1 (\Omega, \mathcal{F},\PP)$, $0<a_1(\omega)\le a_2(\omega)\le \cdots$, and $\lim_{n\to\infty} a_n(\omega)=a(\omega)$.
By the ergodicity of the metric dynamical system $(\Omega, \mathcal{F},\PP,\{\theta_t\}_{t\in \R})$, we have that for a.e. $\omega\in\Omega$,
$$
\int_\Omega a_n(\omega)d\PP(\Omega)=\lim_{t\to\infty}\frac{1}{t}\int_0^t a_n(\theta_\tau\omega)d\tau\le  \hat{a}_{\sup}(\omega)=\int_\Omega \hat{a}_{\sup}(\omega)d\PP(\omega).
$$
This together with { the} Monotone Convergence Theorem implies that
$$
\int_\Omega a(\omega)d\PP(\omega)=\lim_{n\to\infty}\int_\Omega a_n(\omega)d\PP(\omega)\le \int_{\Omega}\hat a_{\sup}(\omega)d\PP(\omega).
$$
Therefore, $a(\cdot)\in L^1 (\Omega, \mathcal{F},\PP)$, and moreover,
 by the ergodicity of the metric dynamical system $(\Omega, \mathcal{F},\PP,\{\theta_t\}_{t\in \R})$, there are
 { $\hat a\in\R$ and a measurable subset $\Omega_2\subset \Omega$ with
 $\P(\Omega_2)=1$ such that $\theta_t \Omega_2=\Omega_2$ for all $t\in\R$}, and
\begin{equation*}
\hat a=\lim_{t\to\infty}\frac{1}{t}\int_0^t a(\theta_\tau\omega)d\tau{=\lim_{t\to\infty}\frac{1}{t}\int_{-t}^0 a(\theta_\tau\omega)d\tau}= \int_\Omega a(\omega) d\P(\omega)\quad {\rm for}\quad a.e. \,\,\omega\in\Omega.
\end{equation*}
The lemma thus follows {with $\Omega_0=\Omega_1\cap\Omega_2$.}
\end{proof}

\begin{lem}\label{average-mean-lemma}   Suppose that $b\in C(\R, (0,\infty))$  and that $0<\underline{\it b}\leq \overline{\it b}<\infty$,
where
\begin{equation*}
\underline{b}= \liminf_{t-s\to\infty}\frac{1}{t-s}\int_s^t b(\tau)d\tau,\quad \overline{b}=\limsup_{t-s\to\infty}\frac{1}{t-s}\int_s^t b(\tau)d\tau.
\end{equation*}
Then
\begin{equation}
\label{b-B-eq}
\underline{\it b}=\sup_{B\in W^{1,\infty}_{\rm loc}(\R)\cap L^{\infty}(\R)}{\rm essinf}_{\tau\in\R}(b(\tau)-{\it B'}(\tau)).
\end{equation}
\end{lem}

\begin{proof}  The proof of this lemma follows from a proper modification of the proof of \cite[Lemma 3.2]{NaRo1}.  For the sake of completeness we give a proof here. Let $0<\gamma<\underline{b}$. By $\overline{b}<\infty$, there is $T>0$ such that
\begin{equation}\label{dd-1}
\gamma <\frac{1}{T}\int_s^{s+T}b(\tau)d\tau<2\overline{b}, \qquad \forall s\in\R.
\end{equation}
Define
$$
B(t)=\int_{kT}^{t}\Big( b(\tau)-\varepsilon_k \Big)d\tau, \quad \forall t\in[kT, (k+1)T] \quad \text{where} \quad \varepsilon_k:=\frac{1}{T}\int_{kT}^{(k+1)T}b(\tau)ds, \quad \forall\ k\in\mathbb{Z}.
$$
It is clear that $B\in W^{1,\infty}_{\rm loc}(\R) \cap L^\infty(\R)$
 with
\begin{equation}\label{dd-2}
\varepsilon_k=b(t)-B'(t) \quad \text{for}\  \  t\in  (kT,(k+1)T).
\end{equation}
Furthermore, it follows from \eqref{dd-1} that $\|B\|_{\infty}\leq  2T\overline{b}$
and that $\gamma<\varepsilon_k$ for every $k\in\mathbb{Z}$. Hence \eqref{dd-2} implies that
$$
\gamma  \le \sup_{B\in W^{1,\infty}_{\rm loc}(\R)\cap L^{\infty}(\R)}{\rm essinf }_{t\in\R}(b(t)-B'(t)).
$$
Since $\gamma$ is arbitrarily chosen less than $\underline{b}$ we deduce that
$$ \underline{b}\leq \sup_{B\in W^{1,\infty}_{\rm loc}(\R)\cap L^{\infty}(\R)}{\rm essinf }_{t\in\R}(b(t)-B'(t)).$$

On the other hand for each given $B\in W^{1,\infty}_{\rm loc}(\R)\cap L^{\infty}(\R)$ and $t>s$ we have
$$
\frac{1}{t-s}\int_s^{t}b(\tau)d\tau\geq {\rm essinf }_{\tau\in\R}(b(\tau)-B'(\tau))+ \frac{(B(t)-B(s))}{t-s}\geq {\rm essinf }_{\tau\in\R}(b(\tau)-B'(\tau))- \frac{2\|B\|_{\infty}}{t-s}.
$$
Hence
$$
\underline{b}=\liminf_{t-s\to\infty}\frac{1}{t-s}\int_s^{t}b(\tau)d\tau\geq {\rm essinf }_{\tau\in\R}(b(\tau)-B'(\tau)) \quad \forall B\in W^{1,\infty}_{\rm loc}(\R)\cap L^{\infty}(\R).
$$
This completes the proof of the lemma.
\end{proof}

In the following, let  $b\in C(\R, (0,\infty))$ be given and satisfy  that $0<\underline{b}\leq \overline{b}<\infty$. Consider
\begin{equation}
\label{b-eq1}
u_t=u_{xx}+b(t)u(1-u),\quad x\in\R.
\end{equation}
For given $u_0\in C_{\rm unif}^b(\R)$ with $u_0\ge 0$, let $u(t,x;u_0,b)$ be the solution of \eqref{b-eq1} with $u(0,x;u_0,b)=u_0(x)$.

For every  $0<\mu<\underline{\mu}^*:=\sqrt{\underline{b}}$, $x\in\R$, $t\in\R$ and $\omega\in\Omega$, let
\begin{equation}
\label{b-eq2}
c(t;b,\mu)=\frac{\mu^2+b(t)}{\mu},\quad
C(t;b,\mu)=\int_0^t c(\tau;b,\mu)d\tau,
\end{equation}
and
\begin{equation}
\label{b-eq3}
\phi^{\mu}(t,x;b)=e^{-\mu (x-C(t;b,\mu))}.
\end{equation}
Then the function $ \phi^{\mu}$ satisfies
\begin{equation}\label{b-eq4}
\phi^{\mu}_t=\phi^{\mu}_{xx} +b(t)\phi^{\mu},\quad x\in\R.
\end{equation}

\begin{lem}
\label{lm0-1}
Let
$$
\phi_+^\mu(t,x;b)=\min\{1,\phi^\mu(t,x;b)\}.
$$
Then
$$
u(t,x;\phi_+^\mu(0,\cdot;b),b)\le \phi_+^\mu(t,x;b)\quad \forall\,\, t>0,\,\, x\in\R.
$$
\end{lem}

\begin{proof}
It follows directly from  the comparison principle for parabolic equations.
\end{proof}

\begin{lem}\label{lm0-2}
   For every $\mu$ with $0<\mu<\tilde{\mu}<\min\{2\mu, \underline{\mu}^*\}$,  there exist  $\{t_k\}_{k\in\Z}$  with $t_k<t_{k+1}$ and $\lim_{k\to\pm\infty}t_k=\pm\infty$,   $B_b\in W^{1,\infty}_{\rm loc}(\R)\cap L^{\infty}(\R)$ with $B_b(\cdot)\in C^1((t_k,t_{k+1}))$ for $k\in\Z$,  and a positive real number $d_b$ such that for every $d\geq d_{b}$ the function
  $$\phi^{\mu,d,B_b}(t,x):=e^{-\mu (x-C(t;b,\mu))}-de^{\big(\frac{\tilde{\mu}}{\mu}-1\big)B_b(t)-\tilde{\mu}(x-C(t;b,\mu))} $$
 satisfies
 $$
 \phi^{\mu,d,B_b}_t\le \phi^{\mu,d,B_b}_{xx}+b(t)\phi^{\mu,d,B_b}(1-\phi^{\mu,d,B_b})
$$
for  $t\in (t_k,t_{k+1})$, $x\ge C(t,b,\mu)+ \frac{\ln d}{\tilde \mu-\mu}+\frac{ B_b(t)}{\mu},\,\, k\in\Z$.
 \end{lem}

 \begin{proof}
 First of all, for given  $0<\mu<\tilde{\mu}<\min\{2\mu, \underline{\mu}^*\}$, let  $0<\delta\ll 1$ such that
  $(1-\delta)\underline{b}>\tilde{\mu}\mu$. It then  follows from the arguments of  Lemma \ref{average-mean-lemma} that there exist $T>0$
  and  $B_b\in W^{1,\infty}_{\rm loc}(\R)\cap L^{\infty}(\R)$ such that $B_b\in C^1((t_k,t_{k+1}))$, where $t_k=kT$ for $k\in\Z$, and
 \begin{equation*}
\tilde{\mu}\mu \leq (1-\delta)b(t)+B_b'(t)\quad\text{for all}\ t\in (t_k,t_{k+1}),\,\, k\in\Z.
 \end{equation*}

 Next, fix the above $\delta>0$ and $B_b(t)$. Let  $d>1$ to be determined later. Let $\xi(t,x)=x-C(t;b,\mu)$. We have
 \begin{align}\label{a-eq0}
 & \phi^{\mu,d,B_b}_t-\Big( \phi^{\mu,d,B_b}_{xx}+b(t)\phi^{\mu,d,B_b}(1-\phi^{\mu,d,B_b})\Big)\nonumber\\
&= d\Big[-(\frac{\tilde{\mu}}{\mu}-1)B_b'(t)+\tilde{\mu}^2-\tilde{\mu}c(t;b,\mu)+b(t) \Big]e^{(\frac{\tilde{\mu}}{\mu}-1)B_b(t)-\tilde{\mu}\xi(t,x)}\nonumber\\
 &\,\, + b(t)\Big[ e^{-2\mu \xi(t,x)}-2de^{(\frac{\tilde{\mu}}{\mu}-1)B_b(t)-(\mu+\tilde{\mu})\xi(t,x)}+d^2e^{2(\frac{\tilde{\mu}}{\mu}-1)B_b(t)-2\tilde{\mu}\xi(t,x)}\Big]\nonumber \\
 &= d\Big(\frac{\tilde{\mu}}{\mu}-1 \Big)\Big[ \tilde{\mu}\mu-b(t)-B_b'(t) \Big]e^{(\frac{\tilde{\mu}}{\mu}-1)B_b(t)-\tilde{\mu}\xi(t,x)} +b(t)e^{-2\mu \xi(t,x)} \nonumber\\
 &\,\,  -d\Big[2e^{-\mu \xi(t,x)}-de^{\Big(\frac{\tilde{\mu}}{\mu}-1\Big)B_b(t)-\tilde{\mu}\xi(t,x)} \Big]e^{\Big(\frac{\tilde{\mu}}{\mu}-1\Big)B_b(t)-\tilde{\mu}\xi(t,x)}\nonumber\\
 &= d\Big(\frac{\tilde{\mu}}{\mu}-1 \Big)\Big[ \tilde{\mu}\mu-(1-\delta)b(t)-B_b'(t) \Big]e^{(\frac{\tilde{\mu}}{\mu}-1)B_b(t)-\tilde{\mu}\xi(t,x)} \nonumber\\
 &\,\, +\Big[e^{-(2\mu-\tilde{\mu})\xi(t,x)}-d\delta\Big(\frac{\tilde{\mu}}{\mu}-1\Big)e^{\Big(\frac{\tilde{\mu}}{\mu}-1\Big)B_b(t)}
 \Big]a(\theta_t\omega)e^{-\tilde{\mu} \xi(t,x)} \nonumber\\
 &\,\,  +d\Big[-2e^{-\mu \xi(t,x)}+de^{\Big(\frac{\tilde{\mu}}{\mu}-1\Big)B_b(t)-\tilde{\mu}\xi(t,x)} \Big]e^{\Big(\frac{\tilde{\mu}}{\mu}-1\Big)B_b(t)-\tilde{\mu}\xi(t,x)}
 \end{align}
for $t\in (t_k,t_{k+1})$.

 Observe now that
 \begin{equation*}
 d\delta\Big(\frac{\tilde{\mu}}{\mu}-1\Big)e^{\Big(\frac{\tilde{\mu}}{\mu}-1\Big)B_b(t)}\geq 1, \qquad \forall\ d\geq \max\Big\{ \frac{e^{-\Big(\frac{\tilde{\mu}}{\mu}-1\Big)\|B_b\|_{\infty}}}{\delta\Big( \frac{\tilde{\mu}}{\mu}-1\Big)}, e^{\Big( \frac{\tilde{\mu}}{\mu}-1\Big)\|B_b\|_{\infty}} \Big\}.
 \end{equation*}
 For this choice of $d$, if $
 \phi^{\mu,d,B_b}(t,x)\geq 0$, which is equivalent to $\xi(t,x)=x-C(t;b,\mu)\ge \frac{\ln d}{\tilde \mu-\mu}+\frac{ B_b(t)}{\mu}$,
 then $ \xi(t,x)\geq 0 $ and each term in the expression at the right hand side of \eqref{a-eq0} is less or equal to zero.
 The lemma thus follows.
\end{proof}

 Recall that  $u_0^*(x)=1$ for $x< 0$ and $u_0^*(x)=0$ for $x>0$. By \cite[Theorem 1]{KPP}, the solution of \eqref{b-eq1} with initial function $u_0^*$, denoted by $u(t,x;u_0^*,b)$, exists for $t>0$.

\begin{lem}
\label{lm0-3}
Suppose that $u_\epsilon\in C_{\rm unif}^b(\R)$ with $u_\epsilon\ge 0$ and
$\lim_{\epsilon \to 0}\int_{-\infty}^\infty|u_\epsilon(x)-u_0^*(x)|dx =0.$
Then for each $t>0$,
$$
\lim_{\epsilon \to 0} \|u(t,\cdot;u_\epsilon,b)-u(t,\cdot;u_0^*,b)\|_\infty=0.
$$
\end{lem}

\begin{proof}
See \cite[Theorem 8]{KPP}.
\end{proof}

\begin{lem}
\label{lm0-4}
For given $u_i\in C_{\rm unif}^b(\R)$ with $u_i\ge 0$ $(i=1,2)$, if $u_1(x)-u_2(x)$ has exactly one simple zero $x_0$ and $u_1(x)>u_2(x)$ for $x<x_0$ and
$u_1(x)<u_2(x)$ for $x>x_0$, then for any $t>0$, there is
$\xi(t)\in [-\infty,\infty]$ such that
$$
u(t,x;u_1,b)\begin{cases} > u(t,x;u_2,x)\quad x< \xi(t)\cr
> u(t,x;u_2,b)\quad x> \xi(t).
\end{cases}
$$
\end{lem}

\begin{proof}
Let $v(t,x)=u(t,x;u_1,b)-u(t,x;u_2,b)$. Then $v(t,x)$ satisfies
$$
v_t=v_{xx}+q(t,x) v,\quad x\in\R,
$$
where $q(t,x)=b(t)-b(t)(u(t,x;u_1,b)+u(t,x;u_2,b))$.  Note that $v(0,x)$ has exactly one simple zero $x_0$ and $v(0,x)>0$ for $x<x_0$, $v(x)<0$ for $x>x_0$. The lemma then follows from \cite[Theorems A,B]{Ang}.
\end{proof}

Let $x(t,b)$ and $x_+(t,b)$ be such that
$$
u(t,x(t,b);u_0^*,b)=\frac{1}{2}
\quad {\rm and}\quad
u(t,x_+(t,b);\phi_+^\mu(0,\cdot;b),b)=\frac{1}{2}.
$$

\begin{lem}
\label{lm0-5}
For any $t>0$, there holds
\begin{equation}
\label{convergence-eq}
u(t,x+x(t,b);u_0^*,b))\begin{cases} \ge u(t,x+x_+(t,b);\phi_+^\mu(0,\cdot;b),b)\quad x<0\cr
\le u(t,x+x_+(t,b);\phi_+^\mu(0,\cdot;b),b)\quad x>0.
\end{cases}
\end{equation}
\end{lem}

\begin{proof}
First, let $\phi_n(x)=\min\{1-\frac{1}{n},\phi^\mu(0,x;b)\}$. Then
$\lim_{n\to\infty}\phi_n(x)=\phi_+^\mu(0,x;b)$
uniformly in $x\in\R$. Then for any given $t>0$,
$$
u(t,x;\phi_+^\mu(0,\cdot;b),b)=\lim_{n\to\infty} u(t,x;\phi_n,b)
$$
uniformly in $x\in\R$. Let $x_+^n(t,b)$ be such that
$u(t,x_+^n(t,b);\phi_n,b)=\frac{1}{2}.$
We have
$$
\lim_{n\to\infty} x_+^n(t,b)=x_+(t,b).
$$

Next, for given $n\ge 1$, let $u_\epsilon^*(x)$ be a nonincreasing function such that $u_\epsilon^*\in C_{\rm unif}^b(\R)$;
$u_\epsilon^*(x)=1$  for $x\ll -1$ and $u_\epsilon^*(x)=0$ for $x\gg 0$; $u_\epsilon^*(x)-\phi_n(x+h)$ has exactly
one simple zero for any $h\in\R$; and
$$
\lim_{\epsilon\to 0}\int_{-\infty}^\infty|u_\epsilon^*(x)-u_0^*(x)|dx=0.
$$
Let $x_\epsilon(t,b)$ be such that
$$
u(t,x;u_\epsilon^*,b)=\frac{1}{2}.
$$
By Lemma \ref{lm0-4}, for any $t>0$,
$$
u(t,x+x_\epsilon(t,b),b)\begin{cases} >u(t,x+x_+^n(t,b);\phi_n,b)\quad x<0\cr
<u(t,x+x_+^n(t,b);\phi_n,b)\quad x>0.
\end{cases}
$$
By Lemma \ref{lm0-3},  for any $t>0$,
$$
\lim_{\epsilon\to 0}\|u(t,\cdot;u_\epsilon^*,b)-u(t,\cdot;u_0^*,b)\|_\infty=0
\quad {\rm and}\quad
\lim_{\epsilon\to 0} x_\epsilon(t,b)=x(t,b).
$$
Letting $\epsilon\to 0$, we get
$$
u(t,x+x(t,b);u_0^*,b)\begin{cases}
\ge u(t,x+x_+^n(t,b);\phi_n,b)\quad x<0\cr
\le u(t,x+x_+^n(t,b);\phi_n,b)\quad x>0.
\end{cases}
$$
Letting $n\to\infty$, the lemma follows.
\end{proof}

\begin{lem}\label{random-fixed-point-lemma}
Let $F: \R\times\Omega\to\R$ be measurable in $\omega\in\Omega$ and continuous hemicompact in $x\in\R$ (i.e for every $\omega\in\Omega$, $F(\cdot,\omega)$ is continuous in $x$ and any sequence { $\{x_n\}_{n\geq 1}\subset \R$} with $|x_n-F(x_n,\omega)|\to 0$ as $n\to\infty$ has a convergent subsequence). Then $F$ has a deterministic fixed point (i.e there is $X: \Omega
\to\R$ such that $F(X(\omega),\omega)=X(\omega)$) if and only if $F$ has random fixed point (i.e there is a measurable function $X: \Omega
\to\R$ such that $F(X(\omega),\omega)=X(\omega)$).
\end{lem}
\begin{proof}
See \cite[Lemma 4.7]{She4}
\end{proof}

\begin{lem}\label{measurable-inverse-process}
Let $f : \R\times\Omega \to (0,1) $ be a measurable function such that for every  $\omega\in\Omega$ the function $f^{\omega}:=f(\cdot,\omega) : \R \to (0,1)$ is  continuously differentiable and  strictly decreasing. Assume that $\lim_{x\to-\infty}f^{\omega}(x)=1$ and $\lim_{x\to\infty}f^{\omega}(x)=0$ for every $\omega\in\Omega$.  Then for every $a\in (0,1)$ the function $\Omega\ni\omega\mapsto f^{\omega,-1}(a)\in\R$ is measurable, where $f^{\omega,-1}$ denotes the inverse function of $f^{\omega}$.
\end{lem}
\begin{proof} Let $a\in(0,1)$ be given. Note that  for every $\omega\in\Omega$, we have that $f^{\omega,-1}(a)$ is the unique fixed point of the function
$$
\R\ni x\mapsto F(x,\omega):=f(x,w)+x-a.
$$
Note that
$$
|x_n-F(x_n,\omega)|=|f(x_n,w)-a|\to 0 \ \text{as}\ n\to \infty \Rightarrow |x_n-f^{\omega,-1}(a)|\to 0 \ \text{as}\ n\to \infty.
$$
Hence the function $F(x,\omega)$ is hemicompact in $x$. By Lemma \ref{random-fixed-point-lemma},  the function $\Omega\ni\omega\mapsto f^{\omega,-1}(a)$ is measurable. The lemma is thus proved.
\end{proof}

\section{Stability of positive random equilibrium solutions}

 In this section, we
  establish some results about the stability of the positive constant equilibrium solution $u\equiv 1$ of \eqref{Main-eq} (resp.  \eqref{nonautonomous-eq}). We also study the existence and stability of {positive random equilibria of \eqref{real-noise-eq}}.
   The results obtained in this section will play a role in later sections for the
  investigation of spreading speeds and take-over property of solutions of  \eqref{Main-eq} (resp.  \eqref{nonautonomous-eq}).

\subsection{Stability of the positive constant equilibrium solution $u\equiv 1$ of \eqref{Main-eq}}

In this subsection, we  establish some results about the stability of the positive constant equilibrium solution $u\equiv 1$ of \eqref{Main-eq} (resp.  \eqref{nonautonomous-eq}).
Observe that $u(t,x)=v(t,x-C(t;\omega))$ with $C(t;\omega)$ being differential in $t$  solves \eqref{Main-eq} if and only if $v(t,x)$ satisfies
\begin{equation}
\label{moving coordinate -eq1}
v_t=v_{xx}+c(t;\omega)v_x+a(\theta_t\omega)v(1-v),
\end{equation}
where $c(t;\omega)=C'(t;\omega)$.  In this subsection, we also study the stability
   of the positive constant equilibrium solution $u\equiv 1$ of \eqref{moving coordinate -eq1}.

We first prove Theorem \ref{stability of const equi solu thm}.

\begin{proof}[Proof of Theorem \ref{stability of const equi solu thm}]
  First, for given  $u_0\in C^{b}_{\rm uinf}(\R)$ with $\inf_{x\in\mathbb{R}}u_0(x)>0$ and $\omega\in {\Omega}$, let $\underline{u}_0:=\min\{1, \inf_{x\in\mathbb{R}}u_0(x)\}$ and $\overline{u}_0:=\max\{1,\sup_{x\in\mathbb{R}}u_0(x)\}$. By {the comparison principle for parabolic equations}, we have that
\begin{equation}\label{stability of const equi solu eq2}
\underline{u}_0\leq  u(t,x;\underline{u}_0,\omega)\leq \min\{1, u(t,x;u_0,\omega)\},\quad \forall\ x\in\R,\ \forall\ t\geq 0
 \end{equation}
 and
 \begin{equation}\label{stability of const equi solu eq3}
 \max\{1, u(t,x;u_0,\omega)\} \leq  u(t,x;\overline{u}_0,\omega)\leq \overline{u}_0, \quad \forall\ x\in\R,\ \forall\ t\geq 0.
\end{equation}
Since $\underline{u}_0$ and $\overline{u}_0$ {are positive  numbers,} by {the uniqueness of solutions of \eqref{Main-eq}} and its corresponding ODE {with a given initial function}, we have that
$$
u(t,x;\underline{u}_0,\omega)=u(t,0;\underline{u}_0,\omega)\quad \text{and}\quad u(t,x;\overline{u}_0,\omega)=u(t,0;\overline{u}_0,\omega)\quad \forall\ x\in\R, \ \forall t\geq0.
$$

Next, let  $\underline{u}(t)=\Big(\frac{1}{u(t,0;\underline{u}_0,\omega)}-1\Big)e^{\int_0^ta(\theta_s\omega)ds}$ and $\overline{u}(t)=\Big(1-\frac{1}{u(t,0;\overline{u}_0,\omega)}\Big)e^{\int_0^ta(\theta_s\omega)ds}$.  It can be verified directly that
\begin{equation*}
\frac{d }{dt}\underline{u}=\frac{d}{dt}\overline{u}=0, \quad t>0.
\end{equation*}
Hence,
$$
\underline{u}(t)=\underline{u}(0) \quad \text{and} \quad \overline{u}(t)=\overline{u}(0),\quad \forall\ t\geq 0,
$$
which is equivalent to
\begin{equation}\label{stability of const equi solu eq4}
1-u(t,x;\underline{u}_0,\omega)=\underline{u}(0)u(t,x;\underline{u}_0,\omega)e^{-\int_0^ta(\theta_s\omega)ds}
\end{equation}
 and
 \begin{equation}\label{stability of const equi solu eq5}
 u(t,x;\overline{u}_0,\omega)-1=\overline{u}(0)u(t,x;\overline{u}_0,\omega)e^{-\int_0^ta(\theta_s\omega)ds}.
\end{equation}

Now, by \eqref{stability of const equi solu eq2}, \eqref{stability of const equi solu eq3}, \eqref{stability of const equi solu eq4} and \eqref{stability of const equi solu eq5}, we have  that
$$
|u(t,x;u_0,\omega)-1|\leq \overline{u}_0\max\{\overline{u}(0),\underline{u}(0)\}e^{-\int_0^ta(\theta_s\omega)ds}, \quad \forall x\in\R,\ t\geq 0,
$$
which implies that inequality \eqref{stability of const equi solu eq1} holds.
Taking $u_0$ to be a positive constant  with $ 0<u_0<1$, it follows from \eqref{stability of const equi solu eq4} that
$$
u(t,x;\underline{u}_0,\omega)=\frac{1}{1+(\frac{1}{\underline{u}_0}-1)e^{-\int_0^ta(\theta_s\omega)ds}}.
$$
If $\|a(\theta_{\cdot}\omega)\|_{L^{1}(0,\infty)}<\infty$, then $\lim_{t\to\infty}u(t,x;\underline{u}_0,\omega)=\frac{1}{1+(\frac{1}{\underline{u}_0}-1)e^{-\|a(\theta_{\cdot}\omega)\|_{L^{1}(0,\infty)}}}<1$, which completes the proof of the theorem.
\end{proof}

\begin{rk}
\begin{itemize}
\item[(1)]
 Theorem \ref{stability of const equi solu thm} guarantees the exponential stability of the trivial constant equilibrium solution $u\equiv 1$ of \eqref{Main-eq} {with respect to the solutions $u(t,x;u_0,\omega)$ with $\inf_{x\in\mathbb{R}^n}u_0(x)>0$} provided that {\bf (H1)} holds. This result  will be useful in the later sections.

\item[(2)] Let $v(t,x;u_0,\omega)$ be the solution of \eqref{moving coordinate -eq1} with $v(0,x;u_0,\omega)=u_0(x)$. The result in Theorem \ref{stability of const equi solu thm} also holds for $v(t,x;u_0,\omega)$.
\end{itemize}
\end{rk}

Let
\begin{equation*}
\underline{c}(\omega)= \liminf_{t-s\to\infty}\frac{1}{t-s}\int_s^t c(\tau;\omega)d\tau,\quad \overline{c}(\omega)=\limsup_{t-s\to\infty}\frac{1}{t-s}\int_s^t c(\tau;\omega)d\tau.
\end{equation*}
 Next, we prove the following theorem about the stability of $u\equiv 1$.

\begin{tm}
\label{uniform-tail-tm}
Assume {\bf (H1)}.
{Suppose that $v(t,x;\omega)$ with $0<v(t,x;\omega)<1$, is an entire solution of \eqref{moving coordinate -eq1} which is nonincreaing in $x$}. For given $\omega\in\Omega$ with $0<{\underline{c}(\omega)\le\overline{c}(\omega)}<\infty$, if  there is $x^*\in\R$ such that $\inf_{t\in\R} v(t,x^*;\omega)>0$, then $\lim_{x\to -\infty} v(t,x;\omega)=1$ uniformly in $t\in\R$.
\end{tm}

To prove the above theorem, we first prove a lemma.

\begin{lem}
\label{convergence-locally-lm}
Let $u_0,u_n\in C_{\rm unif}^b(\R)$ be such that $0\le u_n(x)\le u_0(x)\le 1$. Let $v(t,x;u_0,\theta_{t_0}\omega)$ (respectively $v(t,x;u_n,\theta_{t_0}\omega)$) denote the solution of \eqref{moving coordinate -eq1} with $\omega$ being replaced by $\theta_{t_0}\omega$ and
with initial function $ u_0$ (respectively $u_n $).   If $\lim_{n\to\infty} u_n(x)=u_0(x)$ locally uniformly in $x\in\R$, then for any fixed $t>0$  with $-\infty<\inf_{t_0\in\R}\int_{0}^{t} c(\tau+t_0;\omega)\le \sup_{t_0\in\R} \int_0^t c(\tau+t_0;\omega)<\infty$, we have
$$
\lim_{n\to\infty} v(t,x;u_n,\theta_{t_0}\omega)=v(t ,x;u_0,\theta_{t_0}\omega)
$$
uniformly in $t_0\in\R$ and locally uniformly in $x\in\R$.
\end{lem}

\begin{proof} Fix $\omega\in\Omega$.
For every $n\ge 1$, the function $v^n(t,x;t_0):=v(t,x;u_0,\theta_{t_0}\omega)-v(t,x;u_n,\theta_{t_0}\omega)$ is non-negative and satisfies
\begin{align*}
v^n_t&=v^n_{xx}{ + c(t+t_0;\omega)v^n_x} +  a(\theta_{t_0+t}\omega)(1-(v(t,x;u_0,\theta_{t_0}\omega)+v(t,x;u_n,\theta_{t_0}\omega) ))v^n\\
&\le v^n_{xx} + c(t+t_0;\omega)v^n_x +  a(\theta_{t_0+t}\omega)v^n.
\end{align*}
It follows that, for every $n\ge 1$, $\tilde v^n(t,x;t_0):=v^n(t,x-\int_{t_0}^{t_0+t}c(\tau;\omega)d\tau);t_0)$ satisfies
$$
\tilde v^n(t,x;t_0)\le \tilde v^n_{xx}+a(\theta_{t+t_0}\omega)\tilde v^n,
$$
By the comparison principle for parabolic equations,
$$
0\le v^n(t,\cdot;t_0)\le   e^{\int_{t_0}^{t_0+t}a(\theta_{\tau}\omega)d\tau}e^{t\Delta} v^n(0,\cdot+\int_0^tc(\tau+t_0)d\tau).
$$
Note that $\lim_{n\to\infty} \big(  e^{\int_{t_0}^{t_0+t}a(\theta_{\tau}\omega)d\tau}e^{t\Delta} v^n(0,\cdot+\int_0^tc(\tau+t_0)d\tau)\big)(x)=0$ locally uniformly in $x\in\R$ and uniformly in $t_0\in\R$. Hence
$\lim_{n\to\infty} v^n(t,x;t_0)=0$ uniformly in $t_0\in\R$ and locally uniformly in $x\in\R$.
\end{proof}

We now prove Theorem \ref{uniform-tail-tm}.

\begin{proof}[Proof of Theorem \ref{uniform-tail-tm}]
Fix $\omega\in\Omega$ with $-\infty<\underline{c}(\omega)\le\overline{c}(\omega)<\infty$ and
assume that there is $x^*\in\R$ such that $\inf_{t\in\R} v(t,x^*;\omega)>0$.

{Consider the constant function $u_0\equiv \inf_{t\in\R} v(t,x^*;\omega)$. We first note from the hypotheses of Theorem \ref{uniform-tail-tm} that $u_0>0$. Next, let $\tilde u_0(\cdot)$ be uniformly continuous, $0\le \tilde u_0(x)\le u_0$,
$\tilde u_0(x)=u_0$ for $x\le x^*-1$, and $\tilde u_0(x)=0$ for $x\ge x^*$. For any $R>0$, it holds that   $\tilde{u}_0(x-n)=u_0$ for every $|x|\le R$ and  $n\geq R+1+|x^*|$. This shows that
$\lim_{n\to\infty} \tilde u_0(x-n)=u_0$
locally uniformly in $x\in\R$.}
By {\bf (H1)} and the arguments of Theorem \ref{stability of const equi solu thm},
$$
\lim_{t\to\infty} v(t,x;u_0,\theta_{t_0}\omega)=1
$$
uniformly in $t_0\in\R$ and $x\in\R$.
Hence, for any $\epsilon>0$, there is $T>0$ such that
$$
-\infty<\inf_{t_0\in\R}\int_{0}^{T} c(\tau+t_0;\omega)\le \sup_{t_0\in\R} \int_0^T c(\tau+t_0;\omega)<\infty
$$
and
$$
1>v(T,x; u_0, \theta_{t_0}\omega)>1-\epsilon\quad \forall\,\, t_0\in\R,\,\, x\in\R.
$$
By Lemma \ref{convergence-locally-lm},  there is $N>1$ such that
$$
1>v(T,0;\tilde u_0(\cdot-N),\theta_{t_0}\omega)>1-2\epsilon\quad \forall \,\, t_0\in\R.
$$
This implies that
$$
1>v({ T},-N;\tilde u_0,\theta_{t_0}\omega)>1-2\epsilon\quad \forall \,\, t_0\in\R.
$$

Note  that
$$
v(t-T,x;\omega)\ge \tilde u_0(x)\quad \forall\, \, t\in\R,\,\, x\in\R
$$
and
$$
v(t,x;\omega)=v(T,x;v(t-T,\cdot),\theta_{t-T}\omega).
$$
Hence
$$
1>v(t,x;\omega)=v(T,x;v(t-T,\cdot),\theta_{t-T}\omega)
>1-2\epsilon\quad \forall\,\, t\in\R,\,\, x\le -N.
$$
The theorem thus follows.
\end{proof}

\subsection{Existence and stability of {positive random equilibria of \eqref{real-noise-eq}}}

In this subsection, {we first study} the existence and stability of  {positive random equilibria} of \eqref{real-noise-eq},
and then show that \eqref{real-noise-eq} can be transferred to \eqref{Main-eq}.

To this end, we  consider the following corresponding ODE,
\begin{equation}
\label{real-noise-ode}
\dot u=u(1+ \xi(\theta_t\omega)-u).
\end{equation}
Throughout this subsection, we assume that {\bf (H3)} holds.
For given $u_0\in\R$, let $u(t;u_0,\omega)$ be the solution of \eqref{real-noise-ode} with $u(0;u_0,\omega)=u_0$.
It is known that
$$
u(t;u_0,\omega)=\frac{u_0 e^{t+\int_0^ t\xi(\theta_\tau\omega)d\tau}}
{1+u_0\int_0^t e^{ s+\int_0^s \xi(\theta_\tau\omega)d\tau}ds}.
$$

\begin{tm}
\label{real-noise-tm1}
$Y(\omega)=\frac{1}{\int_{-\infty}^0 e^{ s+\int_0^s \xi(\theta_\tau\omega)d\tau}ds}$ is a random equilibrium of \eqref{real-noise-ode},
that is, $u(t;Y(\omega),\omega)=Y(\theta_t\omega)$ for $t\in\R$ and $\omega\in\Omega$.
\end{tm}

\begin{proof}
First, we note that
\begin{align*}
u(t;Y(\omega),\omega)&=\frac{Y(\omega) e^{t+\int_0^ t\xi(\theta_\tau\omega)d\tau}}
{1+Y(\omega)\int_0^t e^{ s+\int_0^s \xi(\theta_\tau\omega)d\tau}ds}\\
&=\frac{ e^{t+\int_0^ t\xi(\theta_\tau\omega)d\tau}}
{\int_{-\infty}^0 e^{s +\int_0^ s\xi(\theta_\tau\omega)d\tau}ds+\int_0^t e^{ s+\int_0^s \xi(\theta_\tau\omega)d\tau}ds}\\
&=\frac{ e^{t+\int_0^ t\xi(\theta_\tau\omega)d\tau}}
{\int_{-\infty}^t e^{s +\int_0^ s\xi(\theta_\tau\omega)d\tau}ds}.
\end{align*}
Second, note that
\begin{align*}
Y(\theta_t\omega)&=\frac{1}{\int_{-\infty}^0 e^{ s+\int_0^s \xi(\theta_{t+\tau}\omega)d\tau}ds}=\frac{1}{\int_{-\infty}^t e^{(s-t)+\int_0^{s-t} \xi(\theta_{t+\tau}\omega)d\tau}ds}\\
&=\frac{ e^{t+\int_0^ t\xi(\theta_\tau\omega)d\tau}}
{\int_{-\infty}^t e^{s +\int_0^ s\xi(\theta_\tau\omega)d\tau}ds}.
\end{align*}
Hence $u(t;Y(\omega),\omega)=Y(\theta_t\omega)$ and then $Y(\omega)$ is a random equilibrium of \eqref{real-noise-ode}.
\end{proof}

Observe that $0<Y(\omega)<\infty$. Let $\tilde u=\frac{u}{Y(\theta_t\omega)}$ and drop the tilde. We have
\begin{equation}
\label{real-noise-eq1}
u_t=u_{xx}+Y(\theta_t\omega) u(1-u).
\end{equation}
Clearly, \eqref{real-noise-eq1} is of the form \eqref{Main-eq} with $a(\omega)=Y(\omega)$. {Let
$\hat{Y}_{\inf}(\omega)$ and $\hat{Y}_{\sup}(\omega)$ be defined as in \eqref{a-least-mean} and \eqref{a-largest-mean} with $a(\cdot)$ being replaced by $Y(\cdot)$, respectively.}

\begin{lem}
\label{real-noise-lm2} $Y(\omega)$ satisfies the following properties.
\begin{itemize}
\item[(1)] For  a.e. $\omega\in\Omega$, { $0<{ \inf_{ t\in\R}Y(\theta_t\omega)\leq \sup_{t\in\R}Y(\theta_t\omega)}<\infty$.}

\item[(2)] For  a.e.  $\omega\in\Omega$, $\lim_{t\to\infty} \frac{\ln Y(\theta_t\omega)}{t}=0$.

\item[(3)] For  a.e.  $\omega\in\Omega$,  $\lim_{t\to\infty} \frac{\int_0^t Y(\theta_s\omega)ds}{t}=1$.

\item[(4)] ${\hat{Y}_{\inf}}(\omega)=1+\underline{\xi}>0$, and ${\hat{Y}_{\sup}}(\omega)=1+\overline{\xi}<\infty$ for a.e. $\omega\in\Omega$.
\end{itemize}
\end{lem}

\begin{proof}

(1)  First,  note that
\begin{align}\label{aplication-lem-proof-eq2}
\frac{1}{Y(\theta_t\omega)}=& \int_{-\infty}^{-T}e^{s-\int_s^0\xi(\theta_{\tau+t}\omega)\tau}ds+ \int_{-T}^{0}e^{s-\int_s^0\xi(\theta_{\tau+t}\omega)\tau}ds, \quad \forall T>0, \forall \ t\in\R.
\end{align}
 By (H3),  for every $\lambda\in(0,1)$  and a.e. $\omega\in\Omega$ there is $T_{\lambda}\gg 1$,
$$
\lambda\underline{\xi}\leq \frac{1}{T}\int_{0}^T\xi(\theta_{x+\tau}\omega)d\tau\leq \frac{\overline{\xi}}{\lambda}, \quad\forall\ x\in\R, \forall\ T\geq T_{\lambda}.
$$
It then follows that
$$
\int_{-\infty}^{-T_{\lambda}}e^{(1+\frac{\overline{\xi}}{\lambda})s}ds\le \int_{-\infty}^{-T_\lambda}e^{s-\int_s^0\xi(\theta_{\tau+t}\omega)\tau}ds\le \int_{-\infty}^{-T_{\lambda}}e^{(1+\lambda\underline{\xi})s}ds.
\quad 
$$
That is,
\begin{equation}\label{aplication-lem-proof-eq3}
\frac{e^{-(1+\frac{\overline{\xi}}{\lambda})T_{\lambda}}}{(1+\frac{\overline{\xi}}{\lambda})}\leq \int_{-\infty}^{-T_\lambda}e^{s-\int_s^0\xi(\theta_{\tau+t}\omega)\tau}ds\leq \frac{e^{-(1+\lambda\underline{\xi})T_{\lambda}}}{(1+\lambda\overline{\xi})}.
\end{equation}
The first inequality of \eqref{aplication-lem-proof-eq3} {combined with \eqref{aplication-lem-proof-eq2} yields that}
$$
\frac{1}{Y(\theta_t\omega)}\geq \frac{e^{-(1+\frac{\overline{\xi}}{\lambda})T_{\lambda}}}{(1+\frac{\overline{\xi}}{\lambda})}.
$$
Hence
\begin{equation}
\label{uu-00}
Y(\theta_t\omega)\leq (1+\frac{\overline{\xi}}{\lambda})e^{(1+\frac{\overline{\xi}}{\lambda})T_{\lambda}}, \quad\ \forall \ t\in\R.
\end{equation}

Next, {let $\xi_{\inf}(\omega)=\inf_{t\in\R}\xi(\theta_t\omega)$.}   Observe that
$$
\int_{-T_{\lambda}}^{0}e^{s-\int_s^0\xi(\theta_{\tau+t}\omega)\tau}ds\leq \int_{-T_{\lambda}}^{0}e^{s-\int_s^0\xi_{\inf}(\omega)d\tau}ds=\int_{-T_{\lambda}}^{0}e^{s(1+\xi_{\inf}(\omega))}ds.
$$
This combined with the second inequality in \eqref{aplication-lem-proof-eq3} yield that
\begin{equation}\label{uu-01}
\frac{1}{Y(\theta_t\omega)}\leq \frac{e^{-(1+\lambda\underline{\xi})T_{\lambda}}}{(1+\lambda\overline{\xi})}+\int_{-T_{\lambda}}^{0}e^{s(1+\xi_{\inf}(\omega))}ds, \quad \forall\ t\in\R,\,\, a.e.\, \omega\in\Omega.
\end{equation}
It easily follows from { \eqref{uu-00} and \eqref{uu-01}} that
$$ { 0<\inf_{t\in\R} Y(\theta_t\omega)\le \sup_{t\in\R} Y(\theta_t\omega)<\infty} ,\quad  a.e. \ \omega\in\Omega.$$
The result (1) then follows.

(2) It follows from (1).

(3)
 Note that
$$
\frac{\dot{Y}(\theta_t\omega)}{Y(\theta_t\omega)}=1+\xi(\theta_t\omega) -Y(\theta_t\omega).
$$
Integrating both sides with respect to $t$, we obtain that
\begin{equation}\label{aplication-lem-proof-eq4''}
\frac{1}{t-s}\int_s^tY(\theta_{\sigma}\omega)d\sigma+\frac{\ln(Y(\theta_{t}\omega))
-\ln(Y(\theta_{s}\omega))}{t-s}=1+\frac{1}{t-s}\int_s^t\xi(\theta_{\sigma}\omega)d\sigma.
\end{equation}
The result (3) follows from (2) and the fact that $\lim_{t\to\infty}\frac{1}{t}\int_0^t\xi(\theta_s\omega)ds=0$  for a.e. $\omega\in\Omega$.

(4)  Observe that \eqref{aplication-lem-proof-eq4''}  implies that
\begin{align*}
1+\underline{\xi}\leq & { \hat{Y}_{\inf}(\omega)}+\limsup_{t-s\to\infty}\frac{\ln(Y(\theta_{t}\omega))-\ln(Y(\theta_{s}\omega))}{t-s}
\end{align*}
and
\begin{align*}
1+\underline{\xi}\geq & { \hat{Y}_{\inf}(\omega)}+\liminf_{t-s\to\infty}\frac{\ln(Y(\theta_{t}\omega))-\ln(Y(\theta_{s}\omega))}{t-s}
\end{align*}
 for a.e. $\omega\in\Omega$.
It follows from (1) that
$$
\liminf_{t-s\to\infty}\frac{\ln(Y(\theta_{t}\omega))-\ln(Y(\theta_{s}\omega))}{t-s}
=\limsup_{t-s\to\infty}\frac{\ln(Y(\theta_{t}\omega))-\ln(Y(\theta_{s}\omega))}{t-s}=0\quad  {\rm for}\,\, a.e.\, \omega\in\Omega.
$$
Hence we have that $  {\hat{Y}_{\inf}(\omega)}=1+\underline{\xi}>0$ {for a.e. $\omega\in\Omega$}.
Similar arguments yield that  $ {\hat {Y}_{\sup}(\omega)}=1+\overline{\xi}$ {for a.e. $\omega\in\Omega$}.
\end{proof}

\begin{coro}\label{stability-random-equilibrium-cor}
For given  $u_0\in C^{b}_{\rm uinf}(\R)$ with $\inf_{x}u_0(x)>0$,  for a.e. $\omega\in\Omega$,
\begin{equation*}
\lim_{t\to\infty}\|\frac{u(t,\cdot;u_0,\theta_{t_0}\omega)}{Y(\theta_t\theta_{t_0}\omega)}-1\|_{\infty}=0
\end{equation*}
 uniformly in $t_0\in\R$, where $u(t,x;u_0,\theta_{t_0}\omega)$ is the solution of \eqref{real-noise-eq} with $u(0,x;u_0,\theta_{t_0}\omega)=u_0(x)$.
 \end{coro}

\begin{proof}
It follows from Theorem \ref{stability of const equi solu thm}, Theorem  \ref{real-noise-tm1}, and Lemma \ref{real-noise-lm2}.
\end{proof}

\section{Deterministic and linearly {determinate}  spreading speed interval}

 In this section, we discuss the spreading properties of solutions of \eqref{Main-eq} with nonempty compactly supported initials or front like initials
 and prove Theorems  \ref{spreading-speeds-tm} and \ref{spreading-speeds-tm-0}.

We first prove some lemmas.

\begin{lem}\label{spreeding-speed-lem1}
Let  $\omega\in\Omega_0$. If there is a positive constant $c(\omega)>0$ such that
\begin{equation}\label{spreading-speed-lem1-eq1}
\liminf_{t\to\infty}\inf_{s\in\R,|x|\leq c(\omega)t}u(t,x;u_0,\theta_s\omega)>0,\quad \forall\ u_0\in X_c^+
\end{equation}
then $c^*_{\inf}(\omega)\geq c(\omega)$.
Therefore it holds that
\begin{equation}\label{spreading-speed-lem1-eq3}
c^*_{\inf}(\omega)=\sup\{c\in\R^+\ |\ \liminf_{t\to\infty}\inf_{s\in\R,|x|\leq ct}u(t,x;u_0,\theta_s\omega)>0,\quad \forall\ u_0\in X_c^+ \}.
\end{equation}
\end{lem}

\begin{proof} Let  $\omega\in\Omega_0$ and $c(\omega)$ satisfy \eqref{spreading-speed-lem1-eq1}.  Let $0<c<c(\omega)$ and $u_0\in X_c^+$ be given. Choose $\tilde{c}\in (c, c(\omega))$. It follows from \eqref{spreading-speed-lem1-eq1} that
$$
m_{\tilde{c}}:=\liminf_{t\to\infty}\inf_{s\in\R,|x|\leq \tilde{c}t}u(t,x;u_0,\theta_s\omega)>0.
$$
There is $T\gg 1$ such that
\begin{equation}\label{spreading-tm-proof-eq00}
\frac{m_{\tilde{c}}}{2}\leq \min_{|x|\leq \tilde{c}t}u(t,x;u_0,\theta_s\omega),\quad \forall\ s\in \R,\  t\geq T.
\end{equation}

Suppose by contradiction that there is $(s_n,t_n,x_n)\in\R\times\R^+\times\R$ with $|x_n|\leq ct_n$ for every $n\geq 1$ and $t_n\to\infty$ such that
\begin{equation}\label{spreading-tm-proof-eq01}
0<\delta:=\inf_{n\geq 1}|u(t_n,x_n;u_0,\theta_{s_n}\omega)-1|.
\end{equation}
Let $0<\varepsilon<1$ be fixed.  By {\bf (H1)}, Theorem \ref{stability of const equi solu thm} implies that there is $\tilde{T}_\varepsilon>T$ such that
\begin{equation}\label{spreading-tm-proof-eq02}
\|u(t,\cdot;\frac{m_{\tilde{c}}}{2},\theta_s\omega)-1\|_\infty+\|u(t,\cdot;\|u_0\|_{\infty},\theta_s\omega)-1\|_\infty\leq \varepsilon, \quad \forall t\geq \tilde{T}_\varepsilon, \ \forall\ s\in\R.
\end{equation}
Observe that  $ (\tilde{c}-c)(t_n-\tilde{T}_\varepsilon)-2c\tilde{T}_\varepsilon\to\infty$ as $n\to\infty$. Then there is $n_\varepsilon$ such that
$$ (\tilde{c}-c)(t_n-\tilde{T}_\varepsilon)-2c\tilde{T}_\varepsilon\geq T,\quad \forall \ \ n\geq n_\varepsilon.$$
For every $n\geq n_\varepsilon$, let $u_{0n}\in C^{b}_{\rm unif}(\R)$ with $\|u_{0n}\|_{\infty}\leq \frac{m_{\tilde{c}}}{2}$ and
\begin{equation}\label{spreading-tm-proof-eq03}
u_{0n}(x)=\begin{cases}
\frac{m_{\tilde{c}}}{2},\quad \ |x|\leq (\tilde{c}-c)(t_n-\tilde{T}_\varepsilon)-2c\tilde{T}_\varepsilon,\cr
0,\qquad |x|\geq (\tilde{c}-c)(t_n-\tilde{T}_\varepsilon)-c\tilde{T}_\varepsilon.
\end{cases}
\end{equation}

Since $|x|\leq (\tilde{c}-c)(t_n-\tilde{T}_\varepsilon)-c\tilde{T}_\varepsilon $ implies that $|x+x_n|\leq \tilde{c}(t_n-\tilde{T}_\varepsilon)$ for every $n\geq n_\varepsilon$, it follows from \eqref{spreading-tm-proof-eq00} and \eqref{spreading-tm-proof-eq03} that
$$
u_{0n}(x)\leq u(t_n-\tilde{T}_\varepsilon,x+x_n;u_0,\theta_{s_n}\omega),\quad \forall\ x\in\R, \ \forall\ n\geq n_\varepsilon.
$$
By the comparison principle for parabolic equations, we have
\begin{equation}\label{spreading-tm-proof-eq04}
u(t,x;u_{0n},\theta_{\tilde s_n}\omega)\leq  u(t+t_n-\tilde{T}_\varepsilon,x+x_n; u_0,\theta_{s_n}\omega),\quad \forall \ x\in\R, \ t>0, \ n\geq n_\varepsilon,
\end{equation}
where $\tilde{s}_n=s_n+t_n-\tilde{T}_\varepsilon$.

Observe from the definition of $u_{0n}$ that $u_{0n}(x)\to\frac{m_{\tilde{c}}}{2}$ as $n\to\infty$ locally uniformly in $x\in\R$. It then follows from Lemma \ref{convergence-locally-lm} that for every $t>0$,
\begin{equation}\label{spreading-tm-proof-eq05}
|u(t,x;u_{0n},\theta_{\tilde s_n}\omega)- u(t,x;\frac{m_{\tilde{c}}}{2},\theta_{\tilde s_n}\omega) |\to0\ \text{as}\ n\to\infty \ \text{locally uniformly in}\ x\in\R.
\end{equation}
By \eqref{spreading-tm-proof-eq02}, we have that
$$
1-\varepsilon\leq u(\tilde{T}_\epsilon,x;\frac{m_{\tilde{c}}}{2},\theta_{\tilde s_n}\omega)
,\quad \forall\ x\in\R, \forall \ n\geq 1.$$
This combined with \eqref{spreading-tm-proof-eq04} and \eqref{spreading-tm-proof-eq05} {yields} that
\begin{equation}\label{spreading-tm-proof-eq06}
1-\varepsilon\leq \liminf_{n\to\infty}u(\tilde{T}_\varepsilon,0;u_{0n},\theta_{\tilde s_n}\omega)\leq \liminf_{n\to\infty}u(t_n,x_n;u_{0},\theta_{s_n}\omega).
\end{equation}

On the other hand, since $\|u_0(\cdot+x_n)\|_{\infty}=\|u_0\|_{\infty}$ for every $n\geq 1$, it follows from {the comparison principle} for parabolic equations that
\begin{align*}
u(t_n,x;\|u_0\|_{\infty},\theta_{s_n}\omega)\geq &  u(t_n,x;u_0(\cdot+x_n),\theta_{s_n}\omega),\cr
=& u(t_n,x+x_n;u_0,\theta_{s_n}\omega),\quad \forall \ x\in\R, \ t>0, \ n\geq 1.
\end{align*}
This together  with \eqref{spreading-tm-proof-eq02} implies that
$$
\limsup_{n\to\infty}u(t_n,x_n;u_0,\theta_{s_n}\omega)\leq \limsup_{n\to\infty}\|u(t_n,\cdot;\|u_0\|_{\infty},\theta_{s_n}\omega)\|_{\infty} \leq 1+\varepsilon,
$$
which combined with \eqref{spreading-tm-proof-eq06} yields that
$$
1-\varepsilon\leq \limsup_{n\to\infty}u(t_n,x_n;u_0,\theta_{s_n}\omega)\leq \limsup_{n\to\infty}u(t_n,x_n;u_0,\theta_{s_n}\omega)\leq 1+\varepsilon, \forall \ \varepsilon>0.
$$
Letting $\varepsilon\to0$,  we obtain that
$$
\lim_{n\to\infty}|u(t_n,x_n;u_0,\theta_{s_n}\omega)-1|=0,
$$
which contradicts  to \eqref{spreading-tm-proof-eq01}. Thus we have that
$$
\lim_{t\to\infty}\sup_{s\in\R,|x|\leq ct}|u(t,x;u_0,\theta_{s}\omega)-1|=0, \quad \forall u_0\in X_c^+,\ \forall 0<c<c(\omega).
$$
This implies that $c^*_{\inf}(\omega)\geq c(\omega)$.

Therefore, we have that
$$
 c^{*}_{\inf}(\omega)\geq \sup\{c\in\R^+\ |\ \liminf_{t\to\infty}\inf_{|x|\leq ct, s\in\R}u(t,x;u_0,\theta_s\omega)>0,\quad \forall\ u_0\in X_c^+ \}.
$$
On the other hand, it is clear from the definition of $C^{*}_{\sup}(\omega)$ that
$$
 c^{*}_{\inf}(\omega)\leq
\sup\{c\in\R^+\ |\ \liminf_{t\to\infty}\inf_{|x|\leq ct, s\in\R}u(t,x;u_0,\theta_s\omega)>0,\quad \forall\ u_0\in X_c^+ \}.$$
The lemma is thus proved.
\end{proof}

\begin{lem}\label{spreeding-speed-lem2}
 Let $b>0$ be a positive  number and $v_0\in X_c^+$.
{Let $v(t,x;v_0,b)$ be the solution} of
\begin{equation*}
\begin{cases}
v_t=v_{xx}+bv(1-v), \quad x\in\R\cr
v(0,x)=v_0(x),\quad x\in\R.
\end{cases}
\end{equation*}
 Then
$$
\lim_{t\to\infty}\min_{|x|\leq ct}v(t,x;v_0,b)=1,\quad \forall\ 0<c<2\sqrt{b}.
$$
\end{lem}

\begin{proof}It follows from \cite[Page 66, Corollary 1]{ArWe2}.
\end{proof}

\begin{lem}\label{spreeding-speed-lem3}
Assume {\bf (H1)}. Then for every $\omega\in\Omega_0$,
\begin{equation}\label{spreading-speed-lem3-eq1}
\liminf_{t\to\infty}\inf_{s\in\R,|x|\le ct}u(t,x;u_0,\theta_s\omega)>0, \quad \forall\ 0<c<2\sqrt{\underline{a}}, \ \forall\ u_0\in X_c^+.
\end{equation}
Therefore,
$ c^{*}_{\inf}(\omega)\geq 2\sqrt{\underline{a}}, \quad \forall\ \omega\in\Omega_0.$
\end{lem}
\begin{proof}
First, fix $\omega\in\Omega_0$ and  $u_0\in X_c^+$. Let  $0<c<2\sqrt{\underline{a}}$ be given. Choose $b>c$ and $0<\delta<1$ such that $
c<2\sqrt{b}<2\sqrt{\delta \underline{a}}.$
By the proof of Lemma \ref{average-mean-lemma}, there are $\{t_{k}\}_{k\in\Z}$ with $t_k<t_{k+1}$, $t_{k}\to\pm \infty$ as $k\to\pm\infty$ and  $A\in W^{1,\infty}_{loc}(\R)\cap L^\infty(\R)$ such that $A\in C^1(t_k,t_{k+1})$ for every $k$ and
$$
b\leq \delta a(\theta_t\omega)-A'(t),\quad \text{for}\,\, t\in (t_k,t_{k+1}),\,\, k\in\Z.
$$
Let $
\sigma=\frac{(1-\delta)e^{-\|A\|_{\infty}}}{\|u_0\|_{\infty}+1}$
and { $v(t,x;b)=v(t,x;u_0,b)$.}
By Lemma \ref{spreeding-speed-lem2}, we have that
\begin{equation}\label{spreading-lemma-2-eq1}
\liminf_{t\to\infty}\min_{|x|\leq ct}v(t,x;b)=1.
\end{equation}

Next, for given $s\in\R$, let   $\tilde{v}(t,x;s)=\sigma e^{A(t+s)}v(t,x;b)$.  By { the comparison principle} for parabolic equations, we have  that
$$
0<v(t,x;b)\leq \max\{\|u_0\|_{\infty},1\}<\|u_0\|_{\infty}+1, \quad \forall \ x\in\R, \ t\geq  0.
$$
Hence, {it follows from the definition of} $\sigma$ that
$$
0<\tilde{v}(t,x;s)\leq \sigma e^{\|A\|_\infty}(\|u_0\|_\infty+1)=1-\delta, \quad \forall\ x\in\R,\ \ t\geq 0,\\ s\in\R.
$$
Thus for any $s\in\R$,
\begin{align*}
\tilde{v}_t-\tilde{v}_{xx}-a(\theta_{s+t}\omega)\tilde{v}(1-\tilde{v})=&  \left( A'(s+t)+b(1-v)-a(\theta_{s+t}\omega)(1-\tilde{v})\right)\tilde{v}(t,x)\cr
\leq & \left( A'(s+t)+b(1-v)-\delta a(\theta_{s+t}\omega)\right)\tilde{v}(t,x)\cr
\leq & \left( A'(s+t)+b-\delta a(\theta_{s+t}\omega)\right)\tilde{v}(t,x)\cr
\leq& 0, \quad \ t\in(t_k,t_{k+1})\cap[0,\infty), \ \ x\in\R.
\end{align*}
Note  that
$$\tilde{v}(0,x;s)=\sigma e^{A(s)}u_0(x)\leq u_0(x),\quad \forall\ x\in\R.
$$
By {the comparison principle} for parabolic equations again,  we have that
$$
\sigma e^{-\|A\|_{\infty}}v(t,x,b)\leq \tilde{v}(t,x;s)\leq u(t,x;u_0,\theta_s\omega),\quad \forall\ x\in\R,\ s\in\R, \ t\geq 0.
$$
This combined with \eqref{spreading-lemma-2-eq1} yields that
$$
0<\sigma e^{-\|A\|_{\infty}}\leq \liminf_{t\to\infty}\inf_{s\in\R|x|\leq ct}u(t,x;u_0,\theta_s\omega), \quad \forall\ 0<c<2\sqrt{\underline{a}}.
$$
Hence \eqref{spreading-speed-lem3-eq1} holds. By \eqref{spreading-speed-lem3-eq1} and Lemma \ref{spreeding-speed-lem1},  we have $ c^{*}_{\inf}(\omega)\geq 2\sqrt{\underline{a}}, \quad \forall\ \omega\in\Omega_0.$
\end{proof}

Now, we prove Theorem \ref{spreading-speeds-tm}.

\begin{proof}[Proof of Theorem \ref{spreading-speeds-tm} ]

(i) We first prove
$ c^*_{\sup}(\omega)\leq 2\sqrt{\overline{a}}$ for all $\omega\in\Omega_0.$

 Suppose that ${\rm supp}(u_0)\subset(-R,R)$. For every $\mu>0$, let $C_\mu(t,s)=\int_s^{s+t}\frac{\mu^2+a(\theta_\tau\omega)}{\mu}d\tau$ and { $\phi^\mu(x)=\|u_0\|_{\infty}e^{-\mu ( x-R)} $ and $\tilde{\phi}^{\mu}_{\pm}(t,x;s)=\phi^\mu(\pm x-C_\mu(t,s))$} for every $x\in\R$ and $t\geq 0$. Then
$$
\partial_t\tilde{\phi}_{\pm}^\mu-\partial_{xx}\tilde{\phi}_{\pm}^\mu-a(\theta_{s+t}\omega)\tilde{\phi}_{\pm}^\mu(1-\tilde{\phi}_{\pm}^\mu)
=a(\theta_{s+t}\omega)\left( \tilde{\phi}_{\pm}^\mu\right)^2\geq 0, \ x\in\R,\ t>0.
$$
and
$$
u_0(x)\leq  \tilde{\phi}_{\pm}^\mu(0,x;s), \quad \forall x\in\R,\ \forall\ s\in\R.
$$
By {the comparison principle} for parabolic equations, we have
$$
u(t,x;u_0,\theta_s\omega)\leq \tilde{\phi}_{\pm}^\mu(t,x;s)=\|u_0\|_{\infty}e^{-\mu ({\pm}x-R  - C_\mu(t,s))},\quad \forall\ x,s\in\R, \forall t>0, \forall \mu>0.
$$
This implies that
$$
\limsup_{t\to\infty}\sup_{s\in\R,|x|\ge ct} u(t,x;u_0,\theta_s\omega)=0\quad \forall \,\, \mu>0,\,\, c>\frac{\mu^2+\bar a}{\mu}.
$$

For any $c>\bar c^*=2\sqrt{\bar a}=\inf_{\mu>0}\frac{\mu^2+\sqrt{\bar a}}{\mu}$, choose $\mu>0$ such that
$c>\frac{\mu^2+\sqrt{\bar a}}{\mu}>\bar c^*$. By the above arguments, we  have
$$
\limsup_{t\to\infty}\sup_{s\in\R,|x|\ge ct} u(t,x;u_0,\theta_s\omega)=0.
$$
Hence for any $\omega\in\Omega_0$,
 $c^*_{\sup}(\omega)\leq 2\sqrt{\overline{a}}$.

 Next, we prove that $c_{\sup}^*(\omega)\ge 2\sqrt{\overline{a}}$ for all $\omega\in\Omega_0.$
We prove this by contradiction.

Assume that there is $\omega\in\Omega_0$ such that $c_{\sup}^*(\omega)< 2\sqrt{\overline{a}}$. Then there is $0<\delta<1$  such that
 $$c_{\sup}^*(\omega)< 2\sqrt{\delta \overline{a}}.
 $$
 Note that
 $$
 \limsup_{t-s\to\infty}\frac{1}{t-s}\int_s^t a(\theta_\tau\omega)d\tau=\bar a>\delta \bar a.
 $$
 Then there is $0<\delta^{'}<1$ and $\{t_n\}$, $\{s_n\}$ such that $\lim_{n\to\infty} t_n-s_n=\infty$ and
 \begin{equation}
 \label{new-aux-eq2}
 \delta^{'}\frac{1}{t_n-s_n}\int_{s_n}^{t_n}a(\theta_\tau\omega)d\tau>\delta\bar a.
 \end{equation}

   Choose $c\in (c^*_{\sup}(\omega), 2\sqrt{\delta\overline{a}})$.  Set $L=\frac{2\pi}{\sqrt {4\bar a\delta-c^2}}$ and
 $$
 w^+(x)=e^{-\frac{c}{2}x}\sin \Big(\frac{\sqrt {4\bar a\delta-c^2}}{2}x\Big).
 $$
 Then $w^+(x)$ satisfies
 \begin{equation}
 \label{new-aux-eq3}
 \begin{cases}
 w^+_{xx}+cw^+_x+\bar a\delta w^+=0,\quad 0<x<L,\cr
 w^+(0)=w^+(L)=0,
 \end{cases}
 \end{equation}
 and $0<w^+(x)<1$ for $0<x<L$.

For any given $u_0\in X_c^+$, by the assumption that $c>c_{\rm sup}^*(\omega)$,
\begin{equation}
\label{new-aux-eq4}
\limsup_{t\to\infty}\sup_{s\in\R,|x|\ge ct}u(t,x;u_0,\theta_s\omega)=0.
\end{equation}
Hence there is $T>0$ such hat
\begin{equation*}
u(t,x;u_0,\theta_s\omega)<1-\delta^{'}\quad \forall\,\, t\ge T,\,\, |x|\ge ct,\,\, s\in\R,
\end{equation*}
and then
\begin{equation}
\label{new-aux-eq5}
u(t,x;u_0,\theta_s\omega) (1-u(t,x;u_0,\theta_s\omega))>\delta^{'} u(t,x;u_0,\theta_s\omega)\quad  \forall\,\, t\ge T,\,\, |x|\ge ct,\,\, s\in\R.
\end{equation}
Observe that { $u(t,x;u_0,\theta_s\omega)\geq u(t,x;\frac{u_0}{1+\|u_0\|_{\infty}},\theta_s\omega) $ and
$$
u_t(t,x;\frac{u_0}{1+\|u_0\|_{\infty}},\theta_{s}\omega)\ge u_{xx}(t,x;\frac{u_0}{1+\|u_0\|_{\infty}},\theta_{s}\omega),\quad x\in\R.
$$
This implies that
\begin{equation}
\label{new-aux-eq6}
\alpha:=\inf_{s\in\R,0\le x\le L} u(T,x+cT;u_0,\theta_s\omega)\geq \inf_{s\in\R,0\le x\le L} u(T,x+cT;\frac{u_0}{1+\|u_0\|_{\infty}},\theta_s\omega)>0.
\end{equation}
}

Let $v(t,x;s)=u(t,x+ct;u_0,\theta_{s-T}\omega)$. By \eqref{new-aux-eq5},
\begin{equation*}
v_t\ge v_{xx}+cv_x+\delta^{'} a(\theta_{s-T+t}\omega) v,\quad t\ge T,\,\, x\ge 0.
\end{equation*}
Let $w(t,x;s)=e^{-\int_s ^{s-T+t} \big(\delta^{'} a(\theta_\tau\omega)-\delta \bar a\big)d\tau} v(t,x;s)$. Then
$$
w_t\ge w_{xx}+cw_x+\delta\bar a w,\quad t\ge T,\,\, x\ge 0.
$$
By \eqref{new-aux-eq6} and {the comparison principle} for parabolic equations, we have
\begin{equation*}
v(t,x;s)\ge \alpha  e^{\int_s ^{s-T+t} \big(\delta^{'} a(\theta_\tau\omega)-\delta \bar a\big)d\tau}w^+(x),\quad t\ge T,\, \,  0\le x\le L.
\end{equation*}
This implies that for $0\le x\le L$,
\begin{align}
\label{new-aux-eq7}
u(t_n-s_n+T,x+c(t_n-s_n+T);u_0,\theta_{s_n-T}\omega)&\ge \alpha e^{\int_{s_n} ^{t_n} \big(\delta^{'} a(\theta_\tau\omega)-\delta \bar a\big)d\tau}w^+(x)\nonumber\\
&\ge \alpha  w^+(x)\quad {\rm (by \,\,\, \eqref{new-aux-eq2})}.
\end{align}
 By \eqref{new-aux-eq4},
$$
\limsup_{n\to\infty}\sup_{0\le x\le L} u(t_n-s_n+T,x+c(t_n-s_n+T);u_0,\theta_{s_n-T}\omega)=0,
$$
{which contradicts  \eqref{new-aux-eq7}.} Therefore, $c_{\rm sup}^*(\omega)\ge \bar c^*$ and then $c_{\rm sup}^*(\omega)=\bar c^*$ for any $\omega\in\Omega_0$. (i) thus follows.

(ii) By Lemma \ref{spreeding-speed-lem3}, $c_{\inf}^*(\omega)\ge \underline{c}^*$ for every $\omega\in\Omega_0$. It then suffices to prove that $c_{\rm inf}^*(\omega)\le \underline{c}^*$ for every $\omega\in\Omega_0$. We prove this by contradiction.

Assume that there is $\omega\in\Omega_0$ such that $c_{\rm inf}^*(\omega)>\underline{c}^*$. Choose $c\in (\underline{c}^*,c_{\rm inf}^*(\omega))$ and $\delta>1$ such that $c>2\sqrt {\delta \underline{a}}$. Then
$$
\liminf_{t-s\to\infty}\frac{1}{t-s}\int_s ^t a(\theta_\tau\omega)d\tau<\delta \underline{a}.
$$
Hence there are $\{t_n\}$ and $\{s_n\}$ such that $\lim_{n\to\infty}t_n-s_n=\infty$ and
$$
\frac{1}{t_n-s_n}\int_{s_n}^{t_n}a(\theta_\tau)d\tau<\delta\underline{a}\quad \forall\,\, n\ge 1.
$$

Let $\underline{\mu}=\sqrt{\delta\underline{a}}$. Then
\begin{equation}
\label{new-aux-eq0}
2\sqrt{\delta \underline{a}}=\frac{\delta \underline{a}+\underline{\mu}^2}{\underline{\mu}}<c.
\end{equation}
Choose $u_0\in X_c^+$ such that
$$
0\le u_0(x)<1,\quad u_0(x)\le e^{-\underline{\mu}x}\|u_0\|_\infty\quad \forall\,\, x\in\R.
$$
By the assumption that $c<c_{\rm inf}^*(\omega)$, there is
$T>0$ such that for any $t\ge T$ and $s\in\R$,
$$
\inf_{|x|\le ct}u(t,x;u_0,\theta_s\omega)\ge \|u_0\|_\infty.
$$
This implies that for any $n\ge 1$ with $t_n-s_n\ge T$,
\begin{equation}
\label{new-aux-eq1}
\inf_{|x|\le c(t_n-s_n)}u(t_n-s_n,x;u_0,\theta_{s_n}\omega)\ge \|u_0\|_\infty.
\end{equation}

Observe that $u(t,x;u_0,\theta_{s_n}\omega)$ satisfies
\begin{align*}
u_t=u_{xx}+a(\theta_{s_n+t}\omega) u(1-u)\le u_{xx}+a(\theta_{s_n+t}\omega) u.
\end{align*}
It then follows from {the comparison principle} for parabolic equations that
\begin{align*}
u(t,x;u_0,\theta_{s_n}\omega)\le e^{-\underline{\mu} \Big( x-\frac{1}{\underline{\mu}}\int_{s_n}^{s_n+t} (a(\theta_\tau\omega)+\underline{\mu}^2)d\tau\Big)}\|u_0\|_\infty
\end{align*}
and then for $x=c(t_n-s_n)$, we have
\begin{align*}
u(t_n-s_n,x;u_0,\theta_{s_n}\omega)&\le e^{-\underline{\mu} \Big( x-\frac{1}{\underline{\mu}}\int_{s_n}^{t_n} (a(\theta_\tau\omega)+\underline{\mu}^2)d\tau\Big)}\|u_0\|_\infty\\
&\le e^{-\underline{\mu} \Big(x-\frac{1}{\underline{\mu}}(\delta\underline{a}+\underline{\mu}^2)(t_n-s_n)\Big)}\|u_0\|_\infty\\
&=e^{-\underline{\mu} \Big( c-\frac{1}{\underline{\mu}}(\delta\underline{a}+\underline{\mu}^2)\Big)(t_n-s_n)}\|u_0\|_\infty \\
&<\|u_0\|_\infty\qquad\qquad {\rm (by \eqref{new-aux-eq0})},
\end{align*}
which contradicts to  \eqref{new-aux-eq1}.
Therefore $c_{\inf}^*(\omega)\le \underline{c}^*$ for any $\omega\in\Omega_0$ and (ii)  follows.
\end{proof}

The following corollary follows directly from Lemma \ref{real-noise-lm2}  and Theorem \ref{spreading-speeds-tm}.

 \begin{coro}
\label{spreading-cor}
Assume {\bf (H3)}.
  Let $Y(\omega)$ be the random equilibrium solution of \eqref{real-noise-eq} given in \eqref{random-equilibrium-1}. Then
  for any $u_0\in X_c^+$,
 \begin{equation*}
\limsup_{t\to\infty}\sup_{s\in\R,|x|\le ct}|\frac{u(t,x;u_0,\theta_s\omega)}{Y(\theta_{t+s}\omega)}-1|=0, \quad \forall\ 0<c<2\sqrt{1+\underline{\xi}}
\end{equation*}
and
\begin{equation*}
\limsup_{t\to\infty}\sup_{s\in\R,|x|\ge ct} \frac{u(t,x;u_0,\theta_s\omega)}{Y(\theta_{t+s}\omega)}=0, \quad \forall\ c>2\sqrt {1+\bar \xi}
\end{equation*}
for a.e. $\omega\in\Omega$. where $u(t,x;u_0,\theta_s\omega)$ is the solution of \eqref{real-noise-eq}
with $\omega$ being replaced by $\theta_s\omega$ and $u(0,x;u_0,\theta_s\omega)=u_0(x)$.
\end{coro}
}

Finally, we prove   Theorem \ref{spreading-speeds-tm-0}.

\begin{proof}[Proof of Theorem \ref{spreading-speeds-tm-0}]

(i) It is clear that $\tilde c_{\rm sup}^*(\omega)\ge c_{\rm sup}^*(\omega)=\bar c^*$ for any $\omega\in\Omega_0$. It then suffices to prove that $\tilde c_{\rm sup}^*(\omega)\le \bar c^*$ for any $\omega\in\Omega_0$.

To this end, fix $\omega\in\Omega_0$.  For every $\mu>0$, let $C_\mu(t,s)=\int_s^{s+t}\frac{\mu^2+a(\theta_\tau\omega)}{\mu}d\tau$ and  $\tilde{\phi}^{\mu}_{+}(t,x;s)=e^{-\mu(x-C_\mu(t,s))}$ for every $x\in\R$ and $t\geq 0$.
Note that for any $u_0\in\tilde X_c^+$, there is $M_0>0$ such that
$$
u_0(x)\leq  M_0\tilde{\phi}_{+}^\mu(0,x;s), \quad \forall x\in\R,\ \forall\ s\in\R.
$$
Note also that
$$
\partial_tM_0\tilde{\phi}_{+}^\mu-\partial_{xx}M_0\tilde{\phi}_{+}^\mu-a(\theta_{s+t}\omega)M_0\tilde{\phi}_{+}^\mu(1-M_0\tilde{\phi}_{+}^\mu)
=a(\theta_{s+t}\omega)M_0^2 \big( \tilde{\phi}_{+}^\mu\big)^2\geq 0, \ x\in\R,\ t>0.
$$
Hence, by {the comparison principle} for parabolic equations, we have that
$$
u(t,x;u_0,\theta_s\omega)\leq M_0\tilde{\phi}_{+}^\mu(t,x;s)=M_0e^{-\mu (x- C_\mu(t,s))},\quad \forall\ x,s\in\R, \forall t>0, \forall \mu>0.
$$
This implies that
$$
\limsup_{t\to\infty}\sup_{s\in\R,x\ge ct} u(t,x;u_0,\theta_s\omega)=0\quad \forall \,\, \mu>0,\,\, c>\frac{\mu^2+\bar a}{\mu}.
$$

For any $c>\bar c^*=2\sqrt{\bar a}=\inf_{\mu>0}\frac{\mu^2+\sqrt{\bar a}}{\mu}$, choose $\mu>0$ such that
$c>\frac{\mu^2+\sqrt{\bar a}}{\mu}>\bar c^*$. By the above arguments, we  have
$$
 \limsup_{t\to\infty}\sup_{s\in\R,x\ge ct} u(t,x;u_0,\theta_s\omega)=0.
$$
Hence for any $\omega\in\Omega_0$, we have
 $\tilde c^*_{\sup}(\omega)\leq 2\sqrt{\overline{a}}$. (i) thus follows.

\medskip

(ii) First, it is clear that $\tilde c_{\inf}^*(\omega)\ge c_{\rm inf}^*(\omega)=\underline{c}^*$. It then suffices to prove that
$\tilde c_{\rm inf}^*(\omega)\le \underline{c}^*$ for any $\omega\in\Omega_0$. This can be proved by the similar arguments as those in
Theorem \ref{spreading-speeds-tm} (ii).
\end{proof}

{ The following corollary follows directly from Lemma \ref{real-noise-lm2}  and Theorem \ref{spreading-speeds-tm-0}.

 \begin{coro}
\label{spreading-cor2}
Assume {\bf (H3)}.
  Let $Y(\omega)$ be the random equilibrium solution of \eqref{real-noise-eq} given in \eqref{random-equilibrium-1}. Then
  for any $u_0\in \tilde{X}_c^+$,
 \begin{equation*}
\limsup_{t\to\infty}\sup_{s\in\R,x\le ct}|\frac{u(t,x;u_0,\theta_s\omega)}{Y(\theta_{t+s}\omega)}-1|=0, \quad \forall\ 0<c<2\sqrt{1+\underline{\xi}}
\end{equation*}
and
\begin{equation*}
\limsup_{t\to\infty}\sup_{s\in\R,x\ge ct} \frac{u(t,x;u_0,\theta_s\omega)}{Y(\theta_{t+s}\omega)}=0, \quad \forall\ c>2\sqrt {1+\bar \xi}
\end{equation*}
for a.e. $\omega\in\Omega$. where $u(t,x;u_0,\theta_s\omega)$ is the solution of \eqref{real-noise-eq}
with $\omega$ being replaced by $\theta_s\omega$ and $u(0,x;u_0,\theta_s\omega)=u_0(x)$.
\end{coro}
}

\section{Take-over property}

In this section, we investigate the take-over property of \eqref{Main-eq} and prove Theorem \ref{spreading-speeds-tm-1}. We first prove some lemmas.

 Recall that
$$
u_0^*(x)=\begin{cases} 1,\quad x\le 0\cr
0,\quad x>0
\end{cases}
$$
and  that, for $t>0$,  $x(t,\omega)\in\R$ is such that
$$
u(t,x(t,\omega);u_0^*,\omega)=\frac{1}{2}.
$$
 Note that, by Lemma  \ref{measurable-inverse-process}, for each $t>0$, $x(t,\omega)$ is measurable in $\omega$. Note also that for  $\omega\in\Omega$, the mapping $(t,t_0)\ni (0,\infty)\times \R \to u(t,\cdot;u_0^*,\theta_{t_0}\omega)\in C_{\rm unif}^b(\R)$ is continuous and hence $x(t,\theta_{t_0}\omega)$ is continuous in $(t,t_0)\in (0,\infty)\times \R$.

 Suppose that  {\bf (H1)} holds.
Let $\omega \in\Omega_0$, and $0<\mu<\tilde{\mu}<\min\{2\mu,\underline{\mu}^*\}$ be given,  where $\underline{\mu}^*=\sqrt {\underline{a}}$.
 Let $b(t)=a(\theta_t\omega)$. Put
$$
c(t;\omega,\mu)=c(t;b,\mu),\quad C(t;\omega,\mu)=C(t;b,\mu),
$$
and
$$
A_\omega(t)=B_b(t),\quad d_\omega=d_b,
$$
where  $c(t;b,\mu)$ and $C(t;b,\mu)$ are as in \eqref{b-eq2}, and
 $B_b$ and  $d_b$  are as in  Lemma \ref{lm0-2}. Note that we can choose $d_{\theta_{t_0}\omega}=d_\omega$ and $A_{\theta_{t_0}\omega}(t)=A_\omega(t+t_0)$ for any $t_0\in\R$.  Let
\begin{equation}
\label{x-omega-eq}
x_\omega(t)=C(t;\omega,\mu)+\frac{\ln d_{\omega}+\ln \tilde \mu-\ln \mu}{\tilde \mu-\mu}+ \frac{A_\omega(t)}{\mu}.
\end{equation}
Note that for any given $t\in\R$,
$$
\phi^{\mu,d_\omega,A_\omega}(t,x_\omega(t))=\sup_{x\in\R} \phi^{\mu,d_\omega,A_\omega}(t,x)=e^{-\mu\big(\frac{\ln d_\omega}{\tilde \mu-\mu}+\frac{A_\omega(t)}{\mu}\big)}e^{-\mu \frac{\ln \tilde \mu-\ln \mu}{\tilde\mu-\mu}}\big(1-\frac{\mu}{\tilde \mu}\big).
$$
 We introduce the following function
\begin{equation}\label{Lower -sol}\phi_{-}^{\mu}(t,x;\theta_{t_0}\omega)=\begin{cases}
{ \phi^{\mu,d_{\omega},A_{\theta_{t_0}\omega}}(t,x)},\quad  \text{if}\ x\geq x_{\theta_{t_0}\omega}(t), \\
\phi^{\mu,d_{\omega},A_{\theta_{t_0}\omega}}(t,x_{\theta_{t_0}\omega}(t)),\quad \ \text{if}\ x\leq x_{\theta_{t_0}\omega}(t).
\end{cases}
\end{equation}
It is clear  from Lemma \ref{lm0-2}, and {the comparison principle} for parabolic equations,  that
\begin{equation}
\label{eqq1}
0<\phi_{-}^{\mu}(t,x;\theta_{t_0}\omega)< u(t,x;\phi_{+}^{\mu}(\cdot,x;\theta_{t_0}\omega),\theta_{t_0}\omega)\leq 1, \forall\ t\in\R,\,\, x\in\R,\,\, t_0\in\R.
\end{equation}

\begin{lem}
\label{lm00}
For every $\omega\in\Omega_0$,
$\lim_{x\to -\infty} u(t,x+C(t,\theta_{t_0}\omega,\mu); \phi_+^\mu(0,\cdot;\theta_{t_0}\omega),\theta_{t_0}\omega)=1$ uniformly in $t>0$ and $t_0\in\R$, and
$\lim_{x\to \infty} u(t,x+C(t,\theta_{t_0}\omega,\mu);\phi_+^\mu(0,\cdot;\theta_{t_0}\omega),\theta_{t_0}\omega)=0$ uniformly in $t>0$ and $t_0\in\Omega$.
\end{lem}

\begin{proof} First, it follows from Lemma \ref{lm0-1} that
$$
\sup_{t>0,t_0\in\R}u(t,x+C(t,\theta_{t_0}\omega,\mu); \phi_+^\mu(0,\cdot;\theta_{t_0}\omega),\theta_{t_0}\omega)\leq e^{-\mu x}\to 0 \ \text{as}\ x\to\infty.
$$

Second, define $v(t,x;\theta_{t_0}\omega)=u(t,x+C(t,\theta_{t_0}\omega,\mu); \phi_+^\mu(0,\cdot;\theta_{t_0}\omega),\theta_{t_0}\omega)$ and
$$ x^*=\frac{\ln d_\omega+\ln \tilde{\mu}-\ln{\mu}}{\tilde{\mu}-\mu}-\frac{\|A_{\omega}\|_\infty}{\mu}.$$
 It follows from \eqref{x-omega-eq} and \eqref{eqq1} that
$$
0<(1-\frac{\mu}{\tilde{\mu}})e^{-\mu\left( \frac{\ln d_\omega +\ln\tilde{\mu}-\ln\mu}{\tilde{\mu}-\mu}+\frac{\|A_\omega\|_\infty}{\mu}\right)}\leq \inf_{t>0,t_0\in\R}v(t,x^*; \phi_+^\mu(0,\cdot;\theta_{t_0}\omega),\theta_{t_0}\omega).
$$
Moreover, $x\mapsto v(t,x;\theta_{t_0}\omega)$ is decreasing and
$$
v_t=v_{xx} +c(t;\theta_{t_0}\omega,\mu)v_x +a(\theta_t\theta_{t_0}\omega)v(1-v),
$$
where $c(t;\omega,\mu)=C'(t;\omega,\mu)$. By the arguments of Theorem \ref{uniform-tail-tm}, we have that
$$
v(t,x;\theta_{t_0}\omega)\to 1 \ \text{as}\ x\to -\infty
$$
uniformly in $t>0,t_0\in\R$.
\end{proof}

\begin{lem}
\label{lm0-3-1}
For each $t>0$,
there is $m(t)\leq n(t)\in\R$ such that
$$
m(t)\le x(t,\omega) \le n(t)\quad  \,\, \text{for a.e }\ \omega\in\Omega,
$$
and hence  $x(t,\omega)$ is  integrable in $\omega$.
\end{lem}

\begin{proof} First, let
$$
u_{0n}^*(x)=u_0^*(x-n), \quad x\in\R,\,\,  n\in\mathbb{N}.
$$
We have that $0\leq u_{0n}^*(x)\leq 1$ and $u_{0n}^*(x)\to 1$ as $n\to\infty$. By Lemma \ref{convergence-locally-lm}, for every
 $\omega\in\Omega_0$ and  $t>0$
$$
u(t,x;u_{0n}^*,\theta_{t_0}\omega)\to 1\quad \text{as}\ n\to\infty
$$
uniformly in $t_0\in\R$ and locally uniformly in $x\in\R$. Observe that
$$
u(t,x;u_{0n}^*,\theta_{t_0}\omega)=u(t,x-n;u_{0}^*,\theta_{t_0}\omega)
$$
and the mapping $\R\ni x\mapsto u(t,x;u_{0}^*,\theta_{t_0}\omega) $ is decreasing.
Thus, there is $N(t,\omega)\in\mathbb{N}$ such that
$$
u(t,x;u_{0}^*,\theta_{t_0}\omega)\geq \frac{3}{4}, \quad \forall x\leq -N(t,\omega), \quad \forall \ t_0\in\R.
$$
This implies that
$$
-N(t,\omega)\leq \inf_{t_0\in\R}x(t,\theta_{t_0}\omega).
$$
Let
$$
m(t,\omega):=\inf_{t_0\in\R}x(t,\theta_{t_0}\omega)=\inf_{t_0\in\Q} x(t,\theta_{t_0}\omega).
$$
We have that $\Omega_0\ni\omega\mapsto m(t,\omega)\in\R^+$ is measurable and $m(t,\theta_\tau \omega)=m(t,\omega)$ for any
$\tau\in\R$.  By  the ergodicity of the metric dynamical system $(\Omega_0, \mathcal{F},\{\theta_t\}_{t\in\R})$, we have that
$m(t,\omega)=m(t) \quad{\rm for}\,\,  a.e \ \text{in} \ \omega.$

 Next, let $\tilde{u}_{0n}^*(x)=u_0^*(x+n)$. We have that $0\leq u_{0n}^*(x)\leq 1$ and $\tilde{u}_{0n}^*(x)\to 0$ as $n\to\infty$. By Lemma \ref{convergence-locally-lm} again, for every  $\omega\in\Omega_0$ and  $t>0$,
$$
u(t,x;\tilde{u}_{0n}^*,\theta_{t_0}\omega)\to 0\quad \text{as}\ n\to\infty
$$
uniformly in $t_0\in\R$ and locally uniformly in $x\in\R$. Observe that
$$
u(t,x;\tilde{u}_{0n}^*,\theta_{t_0}\omega)=u(t,x+n;u_{0}^*,\theta_{t_0}\omega)
$$
and the mapping $\R\ni x\mapsto u(t,x;u_{0}^*,\theta_{t_0}\omega) $ is decreasing.
Thus, there is $\tilde{N}(t,\omega)\in\mathbb{N}$ such that
$$
u(t,x;u_{0}^*,\theta_{t_0}\omega)\leq \frac{1}{4}, \quad \forall x\geq \tilde{N}(t,\omega), \quad \forall \ t_0\in\R.
$$
This implies that
\begin{equation}\label{fintiness-eq2}
N(t,\omega)\geq \sup_{t_0\in\R}x(t,\theta_{t_0}\omega)
\end{equation}
Let
$$
n(t,\omega):=\sup_{t_0\in\R}x(t,\theta_{t_0}\omega) =\sup_{t_0\in\Q} x(t,\theta_{t_0}\omega).
$$
By \eqref{fintiness-eq2}, we have that $-\infty< x(t,\omega)\leq n(t,\omega)\leq N(t,\omega)<\infty$. Hence $\Omega_0\ni\omega\mapsto m(t,\omega)\in\R^+$ is measurable and $m(t,\theta_\tau\omega)=m(t,\omega)$ for any $\tau\in\R$.
 By  the ergodicity of the metric dynamical system $(\Omega_0, \mathcal{F},\{\theta_t\}_{t\in\R})$, we have that
$m(t,\omega)=m(t) \quad{\rm for}\,\,  a.e \ \text{in} \ \omega.$
\end{proof}

Let $x_+(t,\omega,\mu)$ be such that
$$
u(t,x+x_+(t,\omega,\mu) + C(t,\omega,\mu);\phi_+^\mu(0,\cdot;\omega),\omega)=\frac{1}{2}.
$$

\begin{lem}
\label{lm1}
For any $t>0$, there holds
\begin{equation}
\label{convergence-eq}
u(t,x+x(t;\omega);u_0^*,\omega))\begin{cases} \ge u(t,x+x_+(t,\omega,\mu) + C(t,\omega,\mu);\phi_+^\mu(0,\cdot;\omega),\omega)\quad x<0\cr
\le u(t,x+x_+(t,\omega,\mu) + C(t,\omega,\mu);\phi_+^\mu(0,\cdot;\omega),\omega)\quad x>0.
\end{cases}
\end{equation}
\end{lem}

\begin{proof}
It follows from Lemma \ref{lm0-5}.
\end{proof}

\begin{lem}\label{lem2}
There is  {$\hat M>0$} such that
$$
 x(t,\omega)+x(s,\theta_t\omega)\le x(t+s,\omega)+{ \hat M}
$$
for all $t,s\ge 0$ and a.e. $\omega\in\Omega$.
\end{lem}

\begin{proof}
First, let $\tilde x(t,\omega)$ and $\tilde x_+(t,\omega)$ be such that
$$
u(t,\tilde x(t,\omega);u_0^*,\omega)=\frac{1}{4}
\quad {\rm and}\quad
u(t,\tilde x_+(t,\omega,\mu) +C(t,\omega,\mu); \phi_+^\mu(0,\cdot;\omega),\omega)=\frac{1}{4},
$$
respectively.
Since the function $x\mapsto u(t, x;u_0,\omega) $ is decreasing, we have
\begin{equation}\label{eq6}
\tilde x(t,\omega)>x(t,\omega).
\end{equation}
 Moreover, for each $t>0$, $\tilde x(t,\omega)$ is measurable in $\omega$, and for each  $\omega\in\Omega$,
$\tilde x(t,\theta_{t_0}\omega)$ is continuous in $(t,t_0)\in(0,\infty)\times \R$.
By Lemma \ref{lm1},
\begin{align}\label{Eq_7}
\tilde x(t,\omega)-x(t,\omega)&\le (\tilde x_+(t,\omega,\mu) -C(t,\omega,\mu))-(x_+(t,\omega,\mu) -C(t,\omega,\mu))\nonumber\\
&=\tilde{x}_+(t,\omega,\mu)-x_+(t,\omega,\mu),\,\, \forall\, t>0.
\end{align}
Let
$$
M(\omega)=\sup_{t>0,t_0\in\R}\big(\tilde{x}(t,\theta_{t_0}\omega)-x(t,\theta_{t_0}\omega)\big) =\sup_{t\in (0,\infty)\cap\Q, t_0\in\Q} \big(\tilde{x}(t,\theta_{t_0}\omega)-x(t,\theta_{t_0}\omega)\big).
$$

Note that
\begin{align*}
\frac{1}{2}=&u(t,x_+(t,\theta_{t_0}\omega,\mu)+C(t,\theta_{t_0}\omega,\mu);\phi^\mu_+(\cdot,\cdot;\theta_{t_0}\omega),\theta_{t_0}\omega),\quad \forall\ t>0, \forall\ t_0\in\R,
\end{align*}
and
\begin{align*}
\frac{1}{4}=u(t,\tilde{x}_+(t,\theta_{t_0}\omega,\mu)+C(t,\theta_{t_0}\omega,\mu);\phi^\mu_+(\cdot,\cdot;\theta_{t_0}\omega),\theta_{t_0}\omega),\quad \forall\ t>0, \forall\ t_0\in\R.
\end{align*}
By Lemma \ref{lm00},  there is a positive constant $K(\omega)$ such that
\begin{equation}\label{x_+bound}
 | {x}_+(t,\theta_{t_0}\omega,\mu)|\leq K(\omega) \quad \text{and} \quad | \tilde{x}_+(t,\theta_{t_0}\omega,\mu)|\leq K(\omega),\quad \forall\ t>0, \forall\ t_0\in\R.
\end{equation}
This combined with \eqref{Eq_7} implies that  $M(\omega)<\infty$.

 Note that the function $\Omega_0\ni\omega\mapsto M(\omega)\in\R^+$ is measurable and invariant. By  the ergodicity of the metric dynamical system $(\Omega_0, \mathcal{F},\{\theta_t\}_{t\in\R})$, we have { that there are an invariant} measurable set $\tilde{\Omega}$ with $\P(\tilde\Omega)=1$ and a positive constant {$\hat M$} such that
\begin{equation}\label{eq8}
M(\omega)={\hat M}, \quad \forall\ \omega\in\tilde{\Omega}.
\end{equation}

Second, note that
$$
u_0^*(x)\le 2 u(t,x+x(t,\omega);u_0^*,\omega)
$$
Hence,
\begin{align*}
u(s,x;u_0^*,\theta_t\omega)&\le u(s,x;2u(t,\cdot+x(t,\omega);u_0^*,\omega),\theta_t\omega)\\
&\le 2u(s,x;u(t,\cdot+x(t,\omega);u_0^*,\omega),\theta_t\omega)\\
&=2 u(s,x+x(t,\omega);u_0^*,\omega).
\end{align*}
This implies that
$$
u(s,x(s,\theta_t\omega)+x(t,\omega);u_0^*,\omega)\ge \frac{1}{4}.
$$
It then follows from \eqref{eq8} that
$$
x(s,\theta_t\omega)+x(t,\omega)\le \tilde x(t+s,\omega)\le x(t+s,\omega)+\hat M.
$$
The lemma follows.
\end{proof}

We now prove Theorem \ref{spreading-speeds-tm-1}.

\begin{proof}[Proof of Theorem \ref{spreading-speeds-tm-1}]
(i)  We first prove that there is $c^*$ such that \eqref{average-speed-eq} holds with $\hat c^*$ being replaced by $c^*$.
 To this end, let $y(t,\omega)=-x(t,\omega)+{ \hat M}$ where {$\hat M$} is given by Lemma \ref{lem2}. Then, by Lemma \ref{lem2}
$$
y(t+s,\omega)=-x(t+s,\omega)+\hat M\le -x(t,\omega)-x(s,\theta_t\omega) +2\hat M = y(t,\omega)+y(s,\theta_t\omega)
$$
a.e in $\omega$. By Lemma \ref{lm0-3-1}, $y(t,\cdot)\in L^{1}(\Omega) $.  It then follows from the subadditive ergodic theorem  that there is $c^*\in\R$ such that
$$
\lim_{t\to\infty}\frac{y(t,\omega)}{t}=c^*\quad {\rm for}\,\, a.e. \,\,  \omega\in\Omega.
$$

 Next, we claim that \eqref{asymptoc-tail-eq1} and \eqref{asymptoc-tail-eq2} hold with $\hat c^*$ being replaced by $c^*$.
In fact,  by \eqref{convergence-eq}, \eqref{x_+bound}, and Lemma \ref{lm00},
\begin{align*}
0\leq \sup_{x\geq (c^*+h)t}u(t,x;u_0^*,\omega)\leq& u(t,(c^*+h)t;u_0*,\omega)\cr
\leq&  u(t,(c^*+h)t-x(t;\omega)+x_+(t,\omega,\mu)+C(t,\omega,\mu);\phi_+^\mu(0,\cdot;\omega),\omega) \cr
 \to & 0\ \ \text{as}\ \ t\to \infty,\ \forall h>0,
\end{align*}
and
\begin{align*}
1\geq \inf_{x\leq (c^*-h)t}u(t,x;u_0,\omega)\geq& u(t,(c^*-h)t;u_0,\omega)\cr
\geq & u(t,(c^*-h)t-x(t;\omega)+x_+(t,\omega,\mu)+C(t,\omega,\mu);\phi_+^\mu(0,\cdot;\omega),\omega) \cr
\to & 1\ \ \text{as}\ \ t\to \infty,\ \forall h>0.
\end{align*}
 Therefore, \eqref{asymptoc-tail-eq1} and \eqref{asymptoc-tail-eq2} hold with $\hat c^*$ being replaced by $c^*$.

 Now, we prove that $c^*=\hat c^*$.  By {the comparison principle} for parabolic equations,
$$
u(t,x;u_0^*,\omega)\leq e^{-\mu(x-\frac{1}{\mu}\int_0^{t}(\mu^2+a(\theta_{\tau}\omega)d\tau))}, \forall\ t, \mu>0, \forall\ x\in\R.
$$
Hence
$$
\frac{1}{2}\leq e^{-\mu(x(t,\omega)-\frac{1}{\mu}\int_0^{t}(\mu^2+a(\theta_{\tau}\omega)d\tau))}, \ \forall\ t, \mu>0.
$$
This implies that
$$
\frac{x(t,\omega)}{t}-\frac{\ln(2)}{t\mu}\leq \frac{1}{t\mu}\int_0^{t}(\mu^2+a(\theta_{\tau}\omega)d\tau).
$$
Letting $t\to\infty$, we obtain that
$$
c^*\leq \frac{\mu^2+\hat{a}}{\mu}, \quad \forall\ \mu>0.
$$
Taking $\mu=\sqrt{\hat{a}}$, we obtain that
$$
c^*\leq \hat c^*=2\sqrt{\hat{a}}.
$$
{It then remains} to prove that
$$
c^*\ge  \hat c^*=2\sqrt{\hat{a}}.
$$
We prove this by contradiction.

Assume that $c^*<  \hat c^*=2\sqrt{\hat{a}}$. Then there are $h>0$ and $0<\delta<1$ such that
$$
c^*<c:=c^*+h<2\sqrt{\delta \hat a}.
$$
By \eqref{asymptoc-tail-eq1}, for a.e. $\omega\in\Omega$,
$$
\lim_{t\to\infty} \sup_{x\ge ct} u(t,x;u_0^*,\omega)=0.
$$
Fix such $\omega$.
There are $0<\delta^{'}< 1$ and $T>0$ such that
$$
\delta^{'}\frac{1}{t}\int_0^t a(\theta_\tau\omega)d\tau>\delta { \hat{a}}
$$
and
$$
u(t,x;u_0^*,\omega)\le 1-\delta^{'}\quad \forall \,\, t\ge T,\,\, x\ge ct.
$$
As in the proof of Theorem \ref{spreading-speeds-tm}(i),  let $L=\frac{2\pi}{\sqrt {4{ \hat{a}}\delta-c^2}}$ and
 $$
 w^+(x)=e^{-\frac{c}{2}x}\sin \Big(\frac{\sqrt {4{ \hat{a}}\delta-c^2}}{2}x\Big).
 $$
 By the similar arguments as those in Theorem \ref{spreading-speeds-tm}(i), we have
 \begin{align*}
 u(t,x+ct;u_0^*,\omega)&\ge \alpha e^{\int_T^t (\delta^{'}a(\theta_\tau\omega)-\delta \hat a)d\tau} w^+(x)\\
 &=\alpha e^{-\int_0^T (\delta^{'}a(\theta_\tau\omega)-\delta\hat a)d\tau}e^{\int_0^ t (\delta^{'}a(\theta_\tau\omega)-\delta \hat a)d\tau} w^+(x)\\
 &\ge \alpha e^{-\int_0^T (\delta^{'}a(\theta_\tau\omega)-\delta\hat a)d\tau} w^+(x)
 \end{align*}
 for $0\le x\le L$ and $t\ge T$, where $\alpha=\sup_{0\le x\le L}u(T,x+cT;u_0^*,\omega)$.
 This implies that
 $$
 \lim_{t\to\infty} \sup_{x\ge ct} u(t,x;u_0^*,\omega)>0,
 $$
 which is a contradiction. Hence   $c^*= \hat c^*=2\sqrt{\hat{a}}$.

(ii) For any given $u_0\in\tilde X_c^+$, there are $0<\alpha\le 1\le \beta$ and $x_-<x_+$ such that
$$
\alpha u_0^*(x+x_+)\le u_0(x)\le \beta u_0^*(x+x_-)\quad \forall\,\, x\in\R.
$$
By {the comparison principle} for parabolic equations, we have
$$
\alpha u(t,x;u_0^*(\cdot+x_+),\omega)\le u(t,x;u_0,\omega)\le \beta u(t,x;u_0^*(\cdot+x_-),\omega)\quad \forall\, t\ge 0,\,\, x\in\R.
$$
This together  with \eqref{asymptoc-tail-eq1} implies that there is a measurable set $\Omega_1\subset\Omega$  with $\P(\Omega_1)=1$ such that
$$
\lim_{t\to\infty} \sup_{x\ge (\hat c^*+h)t} u(t,x;u_0,\omega)=0,\quad \omega\in\Omega_1,\ \forall\ h>0,
$$
and
\begin{equation}\label{zz-zz-zz1}
\liminf_{t\to\infty} \inf_{x\le (\hat c^*-h)t} u(t,x;u_0,\omega)\ge \alpha,\quad  \quad \omega\in\Omega_1,\ \forall\ h>0.
\end{equation}
We claim that
\begin{equation}\label{zz-zz-zz2}
\liminf_{t\to\infty} \inf_{x\le (c^*-h)t}u(t,x;u_0,\omega)=1\quad  {\rm for}\,\,\,\omega\in\Omega_1, \ \forall\ h>0.
\end{equation}
Indeed, let $\omega\in\Omega_1$ and $h>0$ be fixed. Let $\{x_n\}$ and $\{t_n\}$ with $t_n\to\infty$  and $x_n\leq (\hat c^*-h)t_n$ be such that
\begin{equation}\label{zz-zz-zz3}
\liminf_{t\to\infty} \inf_{x\le (\hat c^*-h)t}u(t,x;u_0,\omega)=\lim_{n\to\infty}u(t_n,x_n;u_0,\omega).
\end{equation}
For every $0<\varepsilon\ll \frac{1}{2}$, Theorem \ref{stability of const equi solu thm} implies that there is $T_\varepsilon>0$ such that
\begin{equation}\label{zz-zz-zz4}
1-\varepsilon\leq u(t,x;\frac{\alpha}{2},\theta_s\omega),\quad \forall\ x\in\R,\ s\in\R,\ t\geq T_\varepsilon.
\end{equation}
Consider a sequence of $u_{0n}\in C^b_{\rm unif}(\R)$ satisfying that
$$ u_{0n}(x)=
\begin{cases}\frac{\alpha}{2}, \quad x\leq \frac{1}{2}ht_n-2(\hat c^*-\frac{1}{•2}h)T_{\varepsilon}\cr
0, \qquad x\geq \frac{1}{2}ht_n-(\hat c^*-\frac{1}{2}h)T_{\varepsilon}.
\end{cases}
$$
Note that
$$
x\leq \frac{1}{2}ht_n-(\hat c^*-\frac{1}{2}h)T_{\varepsilon}\Rightarrow x+ x_n\leq (\hat c^*-\frac{1}{2}h)(t_n-T_\varepsilon).
$$
By \eqref{zz-zz-zz1}, there is $N_1\gg 1$  such that
$$
u(t_n-T_\varepsilon, x+x_n;u_0,\omega)\geq u_{0n}(x),\quad \forall\ x\in\R, n\geq N_1.
$$
By the comparison principle for parabolic equations, we then have that
$$
u(t+t_n-T_\varepsilon,x+x_n;u_0,\omega)\geq u(t,x;u_{0n},\theta_{t_n-T_\varepsilon}\omega), \quad \forall x\in\R, \forall\ t\geq 0.
$$
In particular, taking $t=T_\varepsilon$ and $x=0$,  we obtain
\begin{equation}\label{zz-zz-zz5}
u(t_n,x_n;u_0,\omega)\geq u(T_\varepsilon,x;u_{0n},\theta_{t_n-T_\varepsilon}\omega).
\end{equation}
 Note that $u_{0n}(x)\to\frac{\alpha}{2}$ as $n\to\infty$. Letting $t\to\infty$ in \eqref{zz-zz-zz5},  it follows from \eqref{zz-zz-zz4} and Lemma \ref{convergence-locally-lm} that
$$
\lim_{n\to\infty}u(t_n,x_n;u_0,\omega)\geq 1-\varepsilon.
$$
Letting $\varepsilon\to0$ in the last inequality, it follows from \eqref{zz-zz-zz3} that
$$
\liminf_{t\to\infty} \inf_{x\le (\hat c^*-h)t}u(t,x;u_0,\omega)\geq 1, \quad  {\rm for}\,\,\,\omega\in\Omega_1, \ \forall\ h>0.
$$
It is clear that
$$
\liminf_{t\to\infty} \inf_{x\le (\hat c^*-h)t}u(t,x;u_0,\omega)\leq 1, \quad  {\rm for}\,\,\,\omega\in\Omega_1, \ \forall\ h>0.
$$
The Claim thus follows and (ii) is proved.
\end{proof}

{ The following corollary follows directly from Lemma \ref{real-noise-lm2}  and Theorem \ref{spreading-speeds-tm-1}.

 \begin{coro}
\label{spreading-cor2}
Assume {\bf (H3)}.
  Let $Y(\omega)$ be the random equilibrium solution of \eqref{real-noise-eq} given in \eqref{random-equilibrium-1} and let $U^*_0(x;\omega)=Y(\omega)$ for $x<0$ and $U^*_0(x;\omega)=0$ for $x>0$. Then,
  \begin{equation*}
  \lim_{t\to\infty}\frac{X(t,\omega)}{t}=2 \quad \text{for a.e}\ \omega\in\Omega,
  \end{equation*}
where $X(t,\omega)$  is such that $u(t,X(t,\omega);U^*_0(\cdot;\omega),\omega)=\frac{1}{2}Y(\omega)$, and
 \begin{equation*}
\lim_{t\to\infty}\sup_{x\ge (2+h)t}\frac{u(t,x;U^*_0(\cdot;\omega),\omega)}{Y(\theta_{t}\omega)}=0, \quad \forall\ h>0, \ \text{a.e}\ \omega\in\Omega,
\end{equation*}
and
\begin{equation*}
\lim_{t\to\infty}\inf_{x\le (2- h)t}\frac{u(t,x;U^*_0(\cdot;\omega),\omega)}{Y(\theta_{t}\omega)}=1, \quad \forall\ h>0, \ \text{a.e}\ \omega\in\Omega,
\end{equation*}
 where $u(t,x;U^*_0(\cdot;\omega),\omega)$ is the solution of \eqref{real-noise-eq} with  $u(0,x;U^*_0(\cdot;\omega),\omega)=U^*_0(x;\omega)$.
\end{coro}
}

\section{Spreading speeds of nonautonomous Fisher-KPP equations}\label{Sec nonautonomous}

In this section we consider the nonautonomous Fisher-KPP equation \eqref{nonautonomous-eq} and prove Theorem \ref{nonautonomous-thm3}.

\begin{proof}[Proof of Theorem \ref{nonautonomous-thm3}]
First, we  prove \eqref{nonauton-spreading-speed-eq1}. To this end,
 for given   $0<c<2\sqrt{\underline{a}_0}$, choose $b>c$ and $0<\delta<1$ such that $
c<2\sqrt{b}<2\sqrt{\delta \underline{a}_0}.$
By the proof of Lemma \ref{average-mean-lemma}, there are $\{t_{k}\}_{k\in\Z}$ with $t_k<t_{k+1}$, $t_{k}\to\pm \infty$ as $k\to\pm\infty$ and  $A\in W^{1,\infty}_{loc}(\R)\cap L^\infty(\R)$ such that  $A\in C^1(t_k,t_{k+1})$ for every $k$ and
$$
b\leq \delta a_0(t)-A'(t),\quad \text{for}\,\, t\in (t_k,t_{k+1}),\,\, k\in\Z.
$$
Let $
\sigma=\frac{(1-\delta)e^{-\|A\|_{\infty}}}{\|u_0\|_{\infty}+1}$
and $v(t,x;b)$ be the solution of the PDE
$$
\begin{cases}
v_t=v_{xx}+bv(1-v), \quad x\in\R, t>0,\cr
v(0,x)=u_0(x), \quad x\in\R.
\end{cases}
$$
By Lemma \ref{spreeding-speed-lem2},  we have that
\begin{equation}\label{spreading-lemma-2-eq1-1}
\liminf_{t\to\infty}\min_{|x|\leq ct}v(t,x;b)=1.
\end{equation}

For given $s\in\R$, let   $\tilde{v}(t,x;s)=\sigma e^{A(t+s)}v(t,x;b)$.  By the similar arguments to those in Lemma \ref{spreeding-speed-lem3},
 it can be proved  that
$$
\sigma e^{-\|A\|_{\infty}}v(t,x,b)\leq \tilde{v}(t,x;s)\leq u(t,x;u_0,\sigma_s a_0),\quad \forall\ x\in\R,\ s\in\R, \ t\geq 0.
$$
This combined with \eqref{spreading-lemma-2-eq1-1} yields that
$$
0<\sigma e^{-\|A\|_{\infty}}\leq \liminf_{t\to\infty}\inf_{s\in\R, |x|\leq ct}u(t,x;u_0,\sigma_s a_0), \quad \forall\ 0<c<2\sqrt{\underline{a}_0}.
$$
By the arguments in Lemma \ref{spreeding-speed-lem1}, it can be proved that
$$
\lim_{t\to\infty}\inf_{s\in\R,|x|\leq ct}|u(t,x;u_0,\sigma_{s}a_0)-1|=0, \quad \forall u_0\in X_c^+,\ \forall 0<c<2\sqrt{\underline{a}_0}.
$$
\eqref{nonauton-spreading-speed-eq1} then follows.

Next,  we  prove \eqref{nonauton-spreading-speed-eq2}. To this end,  for any given $u_0\in X_c^+$, suppose that ${\rm supp}(u_0)\subset(-R,R)$. For every $\mu>0$, let $C_\mu(t,s)=\int_s^{s+t}\frac{\mu^2+a_0(\tau\omega)}{\mu}d\tau$ and $\phi^\mu(x)=\|u_0\|_{\infty}e^{-\mu ( x-R)} $ and $\tilde{\phi}^{\mu}_{\pm}(t,x;s)=\phi_{\pm}^\mu(\pm x-C_\mu(t,s))$ for every $x\in\R$ and $t\geq 0$. It is not difficult to see that
$$
\partial_t\tilde{\phi}_{\pm}^\mu-\partial_{xx}\tilde{\phi}_{\pm}^\mu-a_0({s+t})\tilde{\phi}_{\pm}^\mu(1-\tilde{\phi}_{\pm}^\mu)
=a_0({s+t})\left( \tilde{\phi}_{\pm}^\mu\right)^2\geq 0, \ x\in\R,\ t>0,
$$
and
$$
u_0(x)\leq  \tilde{\phi}_{\pm}^\mu(0,x;s), \quad \forall x\in\R,\ \forall\ s\in\R.
$$
By {the comparison principle} for parabolic equations, we then have that
$$
u(t,x;u_0,\sigma_s a_0)\leq \tilde{\phi}_{\pm}^\mu(t,x;s)=\|u_0\|_{\infty}e^{-\mu ({\pm}x-R\mp C_\mu(t,s))},\quad \forall\ x,s\in\R, \forall t>0, \forall \mu>0.
$$
This implies that
$$
\limsup_{t\to\infty}\sup_{s\in\R,|x|\ge ct} u(t,x;u_0,\sigma_s a_0)=0\quad \forall \,\, \mu>0,\,\, c>\frac{\mu^2+\bar a_0}{\mu}.
$$
For any $c>2\sqrt{\bar a_0}=\inf_{\mu>0}\frac{\mu^2+\sqrt{\bar a_0}}{\mu}$, choose $\mu>0$ such that
$c>\frac{\mu^2+\sqrt{\bar a_0}}{\mu}>\bar c^*$, we  have
$$
\limsup_{t\to\infty}\sup_{s\in\R,|x|\ge ct} u(t,x;u_0,\theta_s\omega)=0.
$$
\eqref{nonauton-spreading-speed-eq2} then follows.
\end{proof}

 We conclude this section with some example of explicit function $a_0(t)$ satisfying {\bf (H2)}.

 \medskip

 Define the sequences $\{l_{n}\}_{n\geq 0}$ and $\{L_n\}_{n\geq 0}$ inductively by
 \begin{equation}\label{ln-Ln-def}
 l_0=0, \quad L_n=l_n+\frac{1}{2^{2(n+1)}}, \quad l_{n+1}=L_n+n+1, \ \ n\geq 0.
 \end{equation}
 Define $a_0(t)$ such that $a_0(-t)=a_0(t)$ for $t\in\R$ and
 \begin{equation}\label{a(t)-def}
 a_0(t)=\begin{cases}
 f_n(t)\qquad \text{if}\ t\in [l_n,L_n]\cr
 g_n(t) \qquad\ \text{if}\ t\in[L_n,l_{n+1}]\cr
 \end{cases}
 \end{equation}
 for $n\ge 0$,
 where $g_{2n}(t)=1$ and $g_{2n+1}(t)=2$ for $n\ge 0$, and $f_0(t)=1$, for $n\ge 1$, $f_{n}$ is H\"older's continuous on  $[l_{n},L_{n}]$, $f_n(l_n)=g_n(l_n)$, $f_{n}(L_{n})=g_n(L_{n})$, and satisfies
 $$
1\leq f_{2n}(t)\leq 2^n, \quad { \max_{t\in [l_{2n},L_{2n}]}}f_{2n}(t)=2^n, 
 $$
 and
 $$
\frac{1}{2^{n+1}}\leq f_{2n+1}(t)\leq 2, \quad {  \min_{t\in [l_{2n+1},L_{2n+1}]}}f_{2n+1}(t)=2^{-(n+1)}. 
 $$
 It is clear  that $a_0(t)$ is locally H\"older's continuous, $
\inf_{t\in\R} a_{0}(t)=0$, and  $\sup_{t\in\R} a_{0}(t)=\infty$. Moreover, it can be verified that
$$
\underline{a_0}=1\quad  \text{and} \quad \overline{a}_0=2.
$$
Hence $a_0(t)$ satisfies {\bf (H2)}.

\medskip


\end{document}